\documentclass{sinst}

\usepackage{amsbsy,amsfonts,amsmath,amssymb,amsthm}
\usepackage{array}
\usepackage{bm}
\usepackage{color}
\usepackage{enumerate}
\usepackage{mathrsfs}
\usepackage{latexsym}
\usepackage[colorlinks=true]{hyperref}
\usepackage{mathtools}
\hypersetup{urlcolor=blue, citecolor=blue, linkcolor=blue}

\definecolor{wineRed}{rgb}{0.7,0,0.3}
\definecolor{grandBleu}{rgb}{0,0,0.8}
\definecolor{darkGreen}{rgb}{0,0.4,0}
\definecolor{blueViolet}{rgb}{0.4,0,1.0}
\definecolor{bloodOrange}{rgb}{0.85,0.05,0}
\definecolor{mycolor}{rgb}{0.8,0,0.2}

\usepackage[textsize=small]{todonotes}
\usepackage{cite}

\usepackage{comment}

\DeclareMathAlphabet{\mathpzc}{OT1}{pzc}{m}{it}

\usepackage[textsize=small]{todonotes}
\numberwithin{equation}{section}

\theoremstyle{plain}
\newtheorem{mainThm}{Main Theorem} 
\newtheorem{lemma}{Lemma}[section]

\newtheorem{theorem}{Theorem}

\theoremstyle{definition}
\newtheorem{definition}{Definition}
\newtheorem{rem}{Remark}
\newtheorem{notn}{Notation}
\newtheorem{ex}{Example}

\def\N{\mathbb{N}}
\def\R{\mathbb{R}}

\def\ts{\textstyle}

\begin{document}
\label{page:t}
\thispagestyle{plain}

\title{Elliptic and Pseudo-Parabolic PDE System with Orientation-Adaptive Anisotropy
\vspace{-4ex}
}
\author{Naotaka Ukai
\footnotemark[1]
}
\affiliation{Division of Mathematics and Informatics, 
\\
Graduate School of Science and Engineering, Chiba University, 
\\
1-33, Yayoi-cho, Inage-ku, 263-8522, Chiba, Japan}
\email{24wd0101@student.gs.chiba-u.jp}

\footcomment{
AMS Subject Classification: 
35K70, 
35K59, 
35K61, 
35J62.
\\
Keywords: elliptic equation, pseudo-parabolic equation, orientation-adaptive anisotropy, nonconvex energy functional, well-posedness, time-discretization
}
\footnotetext[1]{This work was supported by JST SPRING, Grant Number JPMJSP2109.}
\maketitle
\vspace{-2ex}

\noindent
{\bf Abstract.}
In this paper, we consider a coupled system of nonlinear elliptic and pseudo-parabolic PDEs arising in anisotropic monochrome image denoising with orientation-adaptation. The system is derived from the minimization process of a nonconvex energy functional. In particular, we focus on the problem of determining the initial data for the orientation variable. In previous studies, a natural procedure for determining such initial data has not been sufficiently clarified. To address this issue, we introduce a formulation in which the time derivative of the orientation variable is removed. This formulation enables the initial orientation data to be determined implicitly within a time-discrete scheme. On the other hand, this formulation weakens the conventional energy-dissipation structure and leads to new difficulties in constructing a stable variational time-evolution process. To overcome this issue, we develop an analysis based on a time-discretization method and establish the well-posedness of the proposed system, namely existence, uniqueness, and continuous dependence, as well as an energy-inequality. We also show that the proposed time-discrete scheme determines the initial orientation data consistently with the continuous model. These results provide a mathematical framework for the initial-orientation determination problem in orientation-adaptive anisotropic models.

\section{Introduction}
Let $ 0 < T < \infty $ be a fixed constant of time, and $\Omega \subset \mathbb{R}^2$ be a bounded domain with Lipschitz boundary $\Gamma := \partial \Omega$. We denote by $Q := (0, T) \times \Omega$ the space-time cylindrical domain, and by $\Sigma := (0, T) \times \Gamma$ its lateral boundary.

In this paper, we consider a coupled system of nonlinear elliptic and pseudo-parabolic partial differential equations with orientation-adaptive anisotropy. The main issue in this paper is the mathematical analysis of the initial-value determination problem for the orientation variable.

Previous studies \cite{AMSU202411, berkels2006cartoon} introduced anisotropic models involving an additional unknown variable $\alpha$ in order to describe local orientations of structures in images. In these models, another unknown function $u$ represents the gray-scale intensity distribution of a monochrome image, while the anisotropy function $\gamma$ and the orientation variable $\alpha$ describe anisotropic diffusion adapted to polygonal structures.

The system studied in this paper is derived from the minimization process of the following nonconvex energy functional:
\begin{gather*}
    E: [\alpha,u] \in H_0^1(\Omega)\times W^{1, p}_0(\Omega) \mapsto E(\alpha,u) := \frac{\kappa}{2} \int_\Omega |\nabla \alpha|^2 \, dx +\frac{\nu}{p}\int_\Omega |\nabla u|^p \, dx 
    \nonumber
    \\
+\int_\Omega \gamma(R(\alpha) \nabla u) \, dx +\frac{\lambda}{2}\int_\Omega |u -u_\mathrm{org}|^2 \, dx \in [0, \infty).
\end{gather*}
In this context, $\kappa,\mu,\nu,\lambda$ are fixed positive constants, and $p>2$ is a fixed exponent. Moreover, $\gamma$ denotes a smooth non-Euclidean norm representing anisotropy, and $R(\alpha)$ denotes the rotation matrix with the angle $ \alpha $. $ u_\mathrm{org} \in L^2(\Omega) $ is a fixed function. 

Since the energy $ E = E(\alpha, u) $ is a nonconvex functional, uniqueness of minimizers is not generally guaranteed. Therefore, the choice of initial data plays an important role in the corresponding minimization process. For the image variable $u$, it is natural to use the given image data $u_{\mathrm{org}}$ as the initial value. In contrast, for the orientation variable $\alpha$, it is not easy to prepare reliable initial data. However, previous studies have not provided a sufficiently natural and explicit procedure for determining the initial orientation data.

As a mathematical framework for this issue, we introduce a formulation in which the time derivative term of $\alpha$ is removed. More precisely, we consider the following system obtained from the variational system proposed in \cite{AMSU202411} by removing the time-derivative of $\alpha$:
\begin{align*}
    \mbox{(S):} ~~~~&
    \nonumber
    \\
    & 
    \begin{cases}
        -\kappa \mathit{\Delta} \alpha +\nabla \gamma(R(\alpha)\nabla u) \cdot R(\alpha +{\ts \frac{\pi}{2}}) \nabla u \ni 0 \mbox{ in $ Q $,}
        \\[1ex]
        \alpha  =  0 \mbox{ on $ \Sigma $,}
    \end{cases}
    \\
    & 
    \begin{cases}
        \partial_t u -\mathrm{div} \bigl( \, {^\top R} (\alpha) \nabla \gamma(R(\alpha) \nabla u) +\nu |\nabla u|^{p -2} \nabla u +\mu \nabla \partial_t u \, \bigr)
        \\
        \hspace{4ex} +\lambda(u -u_\mathrm{org}) \ni 0 \mbox{ in $ Q  $,}
        \\[1ex]
        u =  \partial_t u = 0 \mbox{ on $ \Sigma $, ~~ $ u(0, x) = u_0(x) $, $ x\in \Omega  $.}
    \end{cases}
\end{align*}

This formulation has the feature that the initial value of $\alpha$ is determined implicitly through the elliptic part of the system within the time-discrete scheme. On the other hand, since the time-evolution structure for $\alpha$ is removed, the conventional energy-dissipation structure becomes weaker. Consequently, standard arguments relying on the full dissipative structure cannot be directly applied.


The essential mathematical difficulty lies in constructing a stable variational time-evolution scheme compatible with both the elliptic equation for $ \alpha $ and the weakened dissipative structure. In view of this difficulty, the Main Theorem of this paper establishes the well-posedness of the system (S), namely, existence, uniqueness, and continuous dependence of solutions, together with an energy-dissipation inequality within the above smooth anisotropy framework. Furthermore, we show that the proposed time-discretization scheme determines the initial orientation data consistently with the continuous model. These assertions constitute the main results of this paper.

The main results are obtained by means of a time-discretization method. The proposed scheme is regarded not only as an analytical approximation method, but also as a constructive procedure for variational time evolution involving the determination of the initial orientation data.

The organization of this paper is as follows. Section 2 presents preliminaries. Section 3 describes the main results, including the Main Theorem and the time-discretization scheme. Section 4 establishes auxiliary results for the time-discretization scheme. Finally, Section 5 proves the Main Theorem.

\section{Preliminaries}
First, throughout this paper, we use the following notations.
\begin{notn}[\bf{Real analysis}]\label{notn1}We define:
  \[a \lor b : = \max\{ a , b \} \mbox{ and } a \land b : = \min\{ a , b \} , \mbox{ for all } a , b \in [-\infty, \infty],\]
  and especially, we note:
  \[[a]^+:= a \lor 0 \mbox{ and } [a]^-:= -(a \land 0), \mbox{ for all } a \in [-\infty,\infty].\]

  Let $d \in \N$ be fixed dimension. We denote by $ | x | $ and $ x \cdot y $ the Euclidean norm of $ x \in \R^d$ and the standard scalar product of $ x , y \in \R^d$, respectively, i.e.,
  \begin{gather*}
    | x | : = \sqrt{x_1^2 + \cdots + x_d^2} \quad \mbox{and} \quad x \cdot y := x_1 y_1 + \cdots + x_d y_d,
    \\
    \mbox{ for all } x = [x_1 , \dots , x_d], \, y = [y_1 , \dots , y_d] \in \R^d.
  \end{gather*}
\end{notn}
In addition, we point out the following elementary fact:
\begin{description}
  \item[\textbf{(Fact 1)}] Let $ m \in \N $ be fixed. If $ \{A_k\}_{k=1}^m \subset \R $ and $ \{ a_n^k \}_{n\in\N}\subset\R $, 
  $ k = 1, \dots , m $ satisfies that:
  \[ \liminf_{ n \rightarrow \infty } a^k_n \geq A_k, \mbox{ for } k = 1, \dots , m, \mbox{ and }\limsup_{ n \rightarrow \infty} \sum_{k=1}^{m} a_n^k \leq \sum_{k=1}^{m} A_k. \]
  Then, $ \lim_{n \rightarrow \infty } a_n^k = A_k $, for $ k = 1 , \dots , m $. 
\end{description}
\begin{notn}[\bf{Abstract functional analysis}]\label{notn2}
  For an abstract Banach space $ X $, the norm of $ X $ is denoted by $ |\cdot|_X $, while the duality pairing between $ X $ and its dual $ X^* $ is represented as $ \langle \cdot, \cdot \rangle_X $. Specifically, when $ X $ is a Hilbert space, its inner product is denoted by $(\cdot, \cdot)_X $.

  For Banach spaces $ X_1 , \dots , X_d $ with $ 1 < d \in \mathbb{N} $, let $ X_1 \times \dots \times X_d $ denote the product Banach space with the norm
  \[| \cdot |_{ X_1\times \dots \times X_d } := | \cdot |_{X_1} + \dots + | \cdot |_{X_d}.\]
  However, when all $ X_1 , \dots , X_d $ are Hilbert spaces, $ X_1 \times \dots \times X_d $ represents the product Hilbert space with the inner product
  \[( \cdot , \cdot )_{ X_1 \times \dots \times X_d } := ( \cdot , \cdot )_{X_1} + \dots + ( \cdot , \cdot )_{X_d},\]
  and the norm 
  \[| \cdot |_{ X_1 \times \dots \times X_d } := \left( | \cdot |_{X_1}^2 + \dots + | \cdot |_{X_d}^2 \right)^{\frac{1}{2}}.\]
  In particular, when all $ X_1 , \dots , X_d $ are identical to a Banach space $ Y $, the product space $ X_1 \times \dots \times X_d $ is simply denoted by $ [Y]^d $.
\end{notn}
\begin{notn}[\bf{Convex analysis}]\label{notn3}
  For any proper lower semi-continuous (denoted by l.s.c., hereafter) and convex function $ \Psi : X \rightarrow (-\infty,\infty]$ on a Hilbert space $ X $, the effective domain of $ \Psi $ is denoted by $ D( \Psi ) $, and the subdifferential of $ \Psi $ is denoted by $ \partial \Psi $. The subdifferential $ \partial \Psi $ is a set-valued mapping that serves as a weak derivative of $ \Psi $. It is characterized as a maximal monotone graph in the product space $ X \times X $. More specifically, for any $ x_0 \in X $, the value $ \partial \Psi( x_0 ) $ is the set of all elements $ x_0^* \in X $ that satisfy the variational inequality
  \[( x_0^*, x - x_0 )_X \leq \Psi (x) - \Psi (x_0), \mbox{ for any } x \in D ( \Psi ), \]
  and the set $ D ( \partial \Psi ) := \{ x \in X \,|\, \partial \Psi (x) \neq \emptyset \}$ is referred to as the domain of $ \partial \Psi $. It is common to use the notation $`` [x_0, x_0^*] \in \partial \Psi $ in $ X \times X "$ to indicate that $`` x_0^* \in \partial \Psi ( x_0 ) $ in $ X $ for $ x_0 \in D ( \partial \Psi ) "$, by identifying the operator $ \partial \Psi $ with its graph in $ X \times X $.
\end{notn}
\begin{ex}\label{ex1}
  Let $\gamma: \R^2 \longrightarrow [0,+\infty)$ be a convex function that belongs to $C^{1,1}(\R^2)$, where the origin $ 0 \in \R^2 $ is the unique minimizer of $ \gamma $. Under these conditions, the following two statements are valid.
   \begin{description}
    \item[(I)] Let $\Phi$ be a functional on $[L^2(\Omega)]^2$, defined as:
  \begin{align*}
    &\Phi:\bm{w}\in[L^2(\Omega)]^2\mapsto\Phi(\bm{w}):=\int_\Omega\gamma(\bm{w})\,dx\in[0,\infty).
  \end{align*}
  Then, $\Phi$ is the proper, l.s.c., and convex function, such that 
  \[D(\Phi)=D(\partial\Phi)=[L^2(\Omega)]^2,\]
  and 
  \begin{equation*}
    \partial\Phi(\bm{w})=\{\nabla\gamma(\bm{w})\},\mbox{ in }[L^2(\Omega)]^2, \mbox{ for any }\bm{w}\in[L^2(\Omega)]^2.  
  \end{equation*}    
  \item[(II)] Let any open interval $I\subset(0,T)$, and let $\widehat{\Phi}^I$ be a functional on
  \\
  $L^2(I;[L^2(\Omega)]^2)\,(=[L^2(I;L^2(\Omega))]^2)$, defined as:
  \begin{align*}
    &\widehat{\Phi}^I:\bm{w}\in L^2(I;[L^2(\Omega)]^2)\mapsto\widehat{\Phi}^I(\bm{w}):=\int_I\Phi(\bm{w}(t))\,dt\in[0,\infty).
  \end{align*}
  Then, $\widehat{\Phi}^I$ is the proper, l.s.c., and convex function, such that 
  \[D(\widehat{\Phi}^I)=D(\partial\widehat{\Phi}^I)=L^2(I;[L^2(\Omega)]^2),\]
  and 
  \begin{align*}
    \partial\widehat{\Phi}^I(\bm{w})
    &=\{\tilde{\bm{w}}^*\in L^2(I;[L^2(\Omega)]^2)~|~\tilde{\bm{w}}^*(t)\in \partial \Phi(\bm{w}(t))\mbox{ in }[L^2(\Omega)]^2,\mbox{ a.e. }t\in I\}
    \\
    &=\{\nabla\gamma(\bm{w})\}, \mbox{ in }L^2(I;[L^2(\Omega)]^2), \mbox{ for any }\bm{w}\in L^2(I;[L^2(\Omega)]^2).
  \end{align*}
  \end{description}
\end{ex}

\begin{definition}[\bf{Mosco-convergence}: cf.\cite{MR0298508}]\label{dfnmosco}
  Let $ X $ be a Hilbert space. Let $ \Psi : X \rightarrow ( -\infty , \infty ] $ be a proper, l.s.c., and convex function, and let $ \{ \Psi_n \}_{ n \in \N } $ be a sequence of proper, l.s.c., and convex functions. Then, we say that $ \Psi_n \to  \Psi $ on $ X $ in the sense of Mosco, if and only if the following two conditions are satisfied:
\begin{description}
  \item[(M1) (Optimality)] For any $w_0 \in D ( \Psi )$, there exists a sequence $ \{w_n\}_{ n \in \N } \subset X $ such that $ w_n \rightarrow w_0 $ in $ X $ and $ \Psi_n ( w_n ) \rightarrow \Psi ( w_0 ) $ as $ n \rightarrow \infty $, 
  \item[(M2) (Lower-bound)] $\liminf_{ n \rightarrow \infty } \Psi_n ( w_n ) \geq \Psi ( w_0 )$ if $w_0 \in X, \{ w_n \}_{ n \in \N } \subset X $, and $w_n \rightarrow w_0 $ weakly in $ X $ as $ n \rightarrow \infty $.
\end{description}
\end{definition}
\begin{definition}[\bf{$ \Gamma $-convergence}; cf.\cite{MR1201152}]\label{dfngamma}
Let $ X $ be a Hilbert space. Let $ \Psi : X \rightarrow ( -\infty , \infty ] $ be a proper and l.s.c. function, and let $ \{ \Psi_n \}_{ n \in \N } $ be a sequence of proper and l.s.c. functions $ \Psi_n : X \rightarrow ( -\infty , \infty ] $, $ n \in \N $. Then, we say that {$ \Psi_n \to \Psi $} on $ X $ in the sense of $ \Gamma $-convergence, if and only if the following two conditions are satisfied:
\begin{description}
    \item[{\boldmath ($\Gamma$1) (Optimality)}] For any $w_0 \in D ( \Psi )$, there exists a sequence $ \{w_n\}_{ n \in \N } \subset X $ such that $ w_n \rightarrow w_0 $ in $ X $ and $ \Psi_n ( w_n ) \rightarrow \Psi ( w_0 ) $ as $ n \rightarrow \infty $, 
    \item[{\boldmath ($\Gamma$2) (Lower-bound)}] $\liminf_{ n \rightarrow \infty } \Psi_n ( w_n ) \geq \Psi ( w_0 )$ if $w_0 \in X, \{ w_n \}_{ n \in \N } \subset X $, and $w_n \rightarrow w_0 $ in $ X $ as $ n \rightarrow \infty $.
\end{description}
\end{definition}

\begin{rem}\label{rem2}
Under the assumption of convexity of functionals, it is important to note that Mosco convergence implies $ \Gamma $-convergence. This means that $ \Gamma $-convergence of convex functions can be interpreted as a weaker form of Mosco convergence. Moreover, regarding the $ \Gamma $-convergence of convex functions, the following observations hold true:
\begin{description}
\item[(Fact 2)] (cf.\cite[Theorem 3.66]{MR0773850} and \cite[Chapter 0]{Kenmochi81}) Let $ X $ be a Hilbert space, and let $ \Psi : X \rightarrow ( -\infty , \infty ] $ and $ \Psi_n : X \rightarrow ( -\infty , \infty ] $, $ n \in \N $, be proper, l.s.c., and convex functions defined on $ X $. Assume that $ \Psi_n \rightarrow \Psi $ on $ X $, in the sense of $ \Gamma $-convergence, as $ n \rightarrow \infty $. Furthermore, let us consider the following assumptions:
\begin{equation*}
  \left\{
  \begin{aligned}
    &[z, z^*] \in X \times X ,~ [z_n, z_n^*] \in \partial \Psi_n \mbox{ in } X \times X,~n \in \N,
    \\
    &z_n \rightarrow z \mbox{ in } X \mbox{ and } z_n^* \rightarrow z^* \mbox{ weakly in } X \mbox{ as } n \rightarrow \infty.  
  \end{aligned}
  \right.
\end{equation*}
Then, it holds that:
\[ [ z , z^* ] \in \partial \Psi \mbox{ in } X \times X , \mbox{ and } \Psi_n ( z_n ) \rightarrow \Psi ( z ) \mbox{ as } n \rightarrow \infty. \]
\item[(Fact 3)](cf.\cite[Lemma 4.1]{MR3661429} and \cite[Appendix]{MR2096945}) Let $ X $ be a Hilbert space, $ d \in \N $ a constant representing the dimension, and $ A \subset \R^d $ a bounded open set. Let $ \Psi : X \rightarrow ( -\infty , \infty ] $ and $ \Psi_n : X \rightarrow ( -\infty , \infty ] $, $ n \in \N $, be proper, l.s.c., and convex functions on a Hilbert space $ X $, where $ \Psi_n \rightarrow \Psi $ on $ X $, in the sense of $ \Gamma $-convergence, as $ n \rightarrow \infty $. Then, we define the sequence $ \{ \widehat{ \Psi }_n^A \}_{ n \in \N } $ of proper, l.s.c., and convex functions on $ L^2 ( A ; X ) $ as follows:
\begin{align*}
  z \in L^2 ( A ; X ) \mapsto \widehat{ \Psi }^A_n ( z ) &: = \left\{
    \begin{aligned}
      & \int_A \Psi _n ( z ( t ) ) \,dt, \mbox{ if } \Psi_n ( z ) \in L^1 ( A ), 
      \\
      & \infty, \mbox{ otherwise}, 
    \end{aligned}
  \right.
\mbox{ for }n \in \N;
\end{align*}
converges to a proper, l.s.c., and convex function $ \widehat{ \Psi }^A $ on $ L^2 ( A ; X ) $, defined as:
\begin{equation*}
  z \in L^2 ( A ; X ) \mapsto \widehat{ \Psi }^A ( z ) : = \left\{
    \begin{aligned}
      & \int_A \Psi ( z ( t ) ) \,dt, \mbox{ if } \Psi ( z ) \in L^1 ( A ), 
      \\
      & \infty, \mbox{ otherwise}, 
    \end{aligned}
  \right.
\end{equation*}
on $ L^2 ( A ; X ) $, in the sense of $ \Gamma $-convergence, as $ n \rightarrow \infty $.
\end{description}
\end{rem}

\begin{notn}\label{deftseq}
  Let $ \tau >0 $ be a constant representing time-step size, and let $\{t_i\}_{i=0}^\infty$ denote a time sequence defined by
  \begin{gather*}
    t_i := i\tau, \mbox{ for any } i = 0, 1, 2, \dots .
  \end{gather*}
  Let $ X $ be a Banach space. For any sequence $ \{ [t_i , u_i] \}_{i=0}^\infty \subset [0,\infty) \times X $, we define three types of interpolations: $ [ \overline{u} ]_\tau,[ \underline{u} ]_\tau \in L^\infty_{\mathrm{loc}}([0,\infty);X) $, and $ [ u ]_\tau \in W^{1,2}_{\mathrm{loc}}([0,\infty);X) $, as follows:
\begin{equation*}
  \left\{
  \begin{aligned}
    &[\overline{u}]_\tau(t):=\chi_{(-\infty,0]}u_{0}+\sum_{i=1}^{\infty}\chi_{(t_{i-1},t_i]}(t)u_{i},
    \\
    &[\underline{u}]_\tau(t):=\sum_{i=0}^{\infty}\chi_{(t_i,t_{i+1}]}(t)u_{i},\hspace*{25ex}\mbox{in }X,\mbox{ for any } t\geq0,
    \\
    &[u]_\tau(t):=\sum_{i=1}^{\infty}\chi_{(t_{i-1},t_i]}(t)\biggl(\frac{t-t_{i-1}}{\tau}u_{i}+
    \frac{t_i-t}{\tau}u_{i-1}\biggr),
  \end{aligned}
  \right.
\end{equation*}
where $ \chi_E : \R \rightarrow \{0,1\} $ represents the characteristic function of the set $ E \subset \R $.
\end{notn}

\section{Main results}
\label{sec:main}

In this paper, the Main Theorem is considered under the following conditions.

\begin{itemize}
  \item[(A0)] The constants $p > 2$, $\nu > 0$, $\mu > 0$, and $\lambda > 0 $ are fixed, and $\kappa > 0$ is a sufficiently large positive constant.
    \vspace{1ex}
  \item[(A1)] A fixed function $ u_{\mathrm{org}} \in L^2 ( \Omega ) $ is given, and it satisfies $ 0 \leq u_\mathrm{org} \leq 1 $ a.e. in $ \Omega $. 
  \item[(A2)] $\gamma: \R^2 \longrightarrow [0,+\infty)$ is a fixed $C^1$-class and convex function satisfying the following conditions: 
      \begin{gather*}
          \operatorname*{arg\,min}_{w \in \mathbb{R}^2} \gamma(w) = \{0\}, \nabla \gamma \in L^\infty(\R^2; \R^2), \mbox{ and } \nabla^2 \gamma \in L^\infty(\R^2; \R^{2\times2}),
      \end{gather*}
      that is, $ \gamma $ and $\nabla \gamma$ are Lipschitz continuous, and the origin $ 0 \in \R^2 $ is the unique minimizer of $ \gamma $.  
  \item[(A3)] $ u_0 \in W^{1,p}_0(\Omega)$ is a fixed initial data such that $0\leq u_0 \leq1$ a.e. in $\Omega$.
\end{itemize}

Next, let us give the definition of the solution to the system (S).

\begin{definition}
A pair of functions $[\alpha,u]\in [L^2(0,T;L^2(\Omega))]^2$ is called a solution to the system (S) if and only if $[\alpha,u]$ fulfills the following conditions:
\begin{itemize}
  \item [(S0)]$\alpha\in W^{1,2}(0,T; H^1_0(\Omega))$, $u\in W^{1,2}(0,T;H^1_0(\Omega))$ $\cap$ $L^\infty(0,T;W^{1,p}_0(\Omega))$ $\subset C(\overline{Q})$, and {$0 \leq u\leq 1 $ in $\overline{Q}$}.
  \item [(S1)]$\alpha$ solves the following variational identity: 
      \begin{gather*}
        \kappa(\nabla\alpha(t),\nabla\varphi)_{[L^2(\Omega)]^2}+\int_{\Omega}\nabla\gamma(R(\alpha(t))\nabla u(t))\cdot R\left(\alpha(t){\ts +\frac{\pi}{2} }\right)\nabla u(t)\varphi \,dx=0,
        \\
        \mbox{ for any }\varphi\in H^1_0(\Omega),\mbox{ and a.e. }t\in(0,T).
      \end{gather*}
  \item[(S2)]$u$ solves the following variational inequality:
    \begin{gather*}
        (\partial_tu(t),u(t)-\psi)_{L^2(\Omega)}+\mu(\nabla\partial_tu(t),\nabla(u(t)-\psi))_{[L^2(\Omega)]^2}
        \\
        +\nu\int_{\Omega}|\nabla u(t)|^{p-2}\nabla u(t)\cdot\nabla(u(t)-\psi)\,dx+\lambda(u(t)-u_\mathrm{org},u(t)-\psi)_{L^2(\Omega)}
        \\
          +\int_{\Omega}\gamma(R(\alpha(t))\nabla u(t))\,dx
        \leq\int_{\Omega}\gamma (R(\alpha(t))\nabla\psi)\,dx,
        \\
        \mbox{ for any }\psi\in W^{1,p}_0(\Omega),\mbox{ and a.e. }t\in(0,T).  
    \end{gather*}
    \item[(S3)] $ u(0) = u_0 $ in $ L^2(\Omega) $.
\end{itemize}
\end{definition}

Now, the Main Theorem is stated as follows.

  \begin{mainThm}(Well-posedness and energy-dissipation)\label{mainThm1}
  Under the assumptions (A0)--(A3), there exists a constant $\kappa_*> 0 $ such that for any $\kappa > \kappa_*$, the system (S) admits a solution $[\alpha,u]$. Additionally, let $ u_{0, k} \in W_0^{1, p}(\Omega) $, $ k = 1, 2 $, be two initial data, and let $ [\alpha_k,u_k] \in [L^2(0, T; L^2(\Omega))]^2 $, $ k = 1, 2 $, be the corresponding solutions to the system (S) with initial data $ u_0 = u_{0, k} $, $ k = 1, 2 $. Let $ J(t) $ be the function of time defined as:
  \begin{gather*}
      J(t) := |(u_1 -u_2)(t)|^2_{L^2(\Omega)} +\mu |\nabla (u_1 -u_2)(t)|_{[L^2(\Omega)]^2}, \mbox{ for all } t \in [0,T].
  \end{gather*}
  Then, there exists a positive constant $ C_* >0 $ such that for any $\kappa > \kappa_*$, 
  \begin{gather}
        \frac{d}{dt}J(t)+\frac{1}{2}(\kappa-\kappa_*)|\nabla(\alpha_1-\alpha_2)(t)|^2_{[L^2(\Omega)]^2}\leq C_*(1+|u_1|_{L^\infty(0,T;W^{1,p}_0(\Omega))})^2J(0),
        \nonumber
        \\
        \mbox{ a.e. $ t \in (0, T) $.}\label{cncl_J}
  \end{gather}
{
  Moreover, the solution $[\alpha,u]$ satisfies the following energy-inequality:
}
  \begin{gather}
    \frac{1}{2}\int_{s}^{t}|\partial_tu(\sigma)|^2_{L^2(\Omega)}\,d\sigma
      +\frac{\mu}{2}\int_{s}^{t}|\nabla\partial_tu(\sigma)|^2_{[L^2(\Omega)]^2}\,d\sigma  
      +E({\alpha}(t),{u}(t))\leq E({\alpha}(s),{u}(s)),
      \nonumber
      \\
      \mbox{for any $ s \in [0, T] $ and any } t \in [s, T].
    \label{ene-inq1}
  \end{gather}
\end{mainThm}

\subsection{Time-discretization scheme}
The Main Theorem is proved by means of the time-discretization method applied to our system (S). 
Let $ m \in \N $ denote the number of subdivisions of the time-interval $ (0, T) $. Let $\tau := \frac{T}{m}$ be the time-step size, and let $\{t_i\}_{i=1}^m$ be a time sequence defined by $t_i := i\tau$, for $i = 1, 2, \dots, m$. 

Based on these, we set up the time-discretization scheme (AP)$^\tau$ as an approximating problem of (S):
\\

(AP)$^\tau$: \, To find $\{[\alpha^{i}, u^{i}]\}_{i=1}^m \subset H^1_0(\Omega) \times W^{1,p}_0(\Omega) $ satisfying: 
  \begin{align*}
    &-\kappa\Delta_0\alpha^{i}
    +\nabla\gamma(R(\alpha^{i})\nabla u^{i})\cdot R\left(\alpha^{i}+{\ts \frac{\pi}{2}}\right)
    \nabla u^{i}=0~\mathrm{in}~L^2(\Omega),\nonumber
    \\
    &\frac{u^{i}-u^{i-1}}{\tau}
    -\mathrm{div}\Bigl( {^\top} R(\alpha^{i})\nabla{\gamma}(R(\alpha^{i})\nabla u^{i})
    +\nu|\nabla u^{i}|^{p-2}\nabla u^{i}+\frac{\mu}{\tau}\nabla(u^{i}-u^{i-1})\Bigr)
    \\
    &\quad 
    +\lambda(u^{i}-u_{\mathrm{org}})
    =0~\mathrm{in}~L^2(\Omega),\nonumber
      ~~\mbox{for }i=1,2,\dots,m,
    \end{align*}
where $ u^{0}=u_0 \in W^{1,p}_0(\Omega) $ is as in (A3), and $\Delta_0$ is the Laplacian operator subject to the zero-Dirichlet boundary condition. Moreover, we give a function $\alpha^0\in H^1_0(\Omega)$ that satisfies the following condition. 
\begin{align}\label{alpha0}
  -\kappa\Delta_0\alpha^0+\nabla\gamma(R(\alpha^{0})\nabla u_0)\cdot R\left(\alpha^{0}+{\ts \frac{\pi}{2}}\right)\nabla u_0=0~\mathrm{in}~L^2(\Omega).
\end{align}
\begin{rem}
The elliptic problem \eqref{alpha0} plays some essential roles in the present analysis. First, the existence of such a function $\alpha^0$ is required in order to obtain an Aubin-type compactness property for the sequence of approximate solutions generated by the time-discretization scheme. 
Secondly, the function $\alpha^0$ is required in the analysis of the energy-dissipation property associated with the energy inequality \eqref{ene-inq1}.
Thirdly, the function $\alpha^0$ corresponds to the initial value of the orientation variable $\alpha$. Hence, the above elliptic equation can be regarded as the determining condition for the initial orientation data discussed in the Introduction. Moreover, we can verify that the solution $\alpha^0$ to \eqref{alpha0} is unique, if $\kappa$ satisfies the following condition:
  \begin{align*}
    \kappa>4\sqrt{2}(1+|\nabla u_0|_{[L^p(\Omega)]^2})^2(1+C_P)^2\Big(C_{H^1}^{L^\frac{2p}{p-2}}+C_{H^1}^{L^\frac{2p}{p-1}}\Big)^2|\nabla\gamma|_{W^{1,\infty}(\R^2;\R^2)}.
  \end{align*}
  Here, let $C_P$  denote the constant in Poincar\'e's inequality, and let $C_{H^1}^{L^\frac{2p}{p-2}}$ and $C_{H^1}^{L^\frac{2p}{p-1}}$ denote the constants associated with the Sobolev embeddings $H^1(\Omega) \subset L^\frac{2p}{p-2}(\Omega)$ and $H^1(\Omega) \subset L^\frac{2p}{p-1}(\Omega)$ in two dimensions, respectively.
\end{rem}

\section{Auxiliary results for time-discretization scheme}

This section is devoted to auxiliary results for the time-discretization scheme (AP)$^\tau$.
The solution to the scheme (AP)$^\tau$ is defined as follows.
\begin{definition}
  A sequence of pair of functions $\{[ \alpha^{i}, u^{i}]\}_{i=1}^m$ is called a solution to (AP)$^\tau$ if $\{[ \alpha^{i}, u^{i}]\}_{i=1}^m \subset H^1_0(\Omega) \times W^{1,p}_0(\Omega)$, and $[ \alpha^{i}, u^{i}]$ fulfills the following items for any $i = 1,2,\dots,m$:
  \\
  (AP1;$\alpha^i$,$u^i$) $\alpha^{i}$ solves the following variational identity:
  \begin{gather*}
    \kappa(\nabla\alpha^{i},\nabla\varphi)_{[L^2(\Omega)]^2}
    +(\nabla{\gamma}(R(\alpha^{i})\nabla u^{i}) \cdot 
    R\left(\alpha^{i}+{\ts \frac{\pi}{2}}\right)\nabla u^{i},\varphi)_{L^2(\Omega)}
    =0,
    \\
    \mbox{ for any }\varphi\in H^1_0(\Omega).
  \end{gather*}
  (AP2;$\alpha^i$,$u^i$) $u^{i}$ solves the following variational identity:
  \begin{align*}
    & \frac{1}{\tau}(u^{i}-u^{i-1},\psi)_{L^2(\Omega)}+(\nabla\gamma(R(\alpha^{i})\nabla u^{i}),R(\alpha^{i})\nabla\psi)_{[L^2(\Omega)]^2}
    \\
    &\quad +\nu\int_\Omega|\nabla u^{i}|^{p-2}\nabla u^{i}\cdot\nabla\psi\,dx
    +\frac{\mu}{\tau}(\nabla(u^{i}-u^{i-1}),\nabla\psi)_{[L^2(\Omega)]^2}
    \\
    &\quad +\lambda(u^{i}-u^{\mathrm{org}},\psi)_{L^2(\Omega)}=0,~~~\mbox{ for any }\psi\in W^{1,p}_0(\Omega).
  \end{align*}
\end{definition}
\begin{theorem}(Existence, uniqueness of solution with energy-inequality)\label{003Thm1}
\\
  There exists a positive constant $\widehat{\kappa} > 0 $, and a sufficiently small constant $\widehat{\tau}\in (0,1)$ such that for any $\kappa > \widehat{\kappa}$, and for any $ \tau \in (0,\widehat{\tau}) $, (AP)$^\tau$ admits a unique solution $\{[ \alpha^{i}, u^{i}]\}_{i=1}^m$. Moreover, the solution $\{[ \alpha^{i}, u^{i}]\}_{i=1}^m$ fulfills the following energy-inequality:
  \begin{gather}
    \frac{1}{2\tau}|u^{i}-u^{i-1}|^2_{L^2(\Omega)}+\frac{\mu}{2\tau}|\nabla(u^{i}-u^{i-1})|^2_{[L^2(\Omega)]^2} +E(\alpha^{i},u^{i}) \leq E(\alpha^{i-1},u^{i-1}),
    \nonumber
    \\
    \mbox{ for any } i=1,2,\dots,m.\label{f-ene0}
  \end{gather}
\end{theorem}

We prepare a lemma for the proof of Theorem \ref{003Thm1}. 

\begin{lemma}\label{003lemma1}
  For $\overline{w}\in H^1_0(\Omega)$, we consider the following elliptic boundary value problem ($E; \alpha,u,\overline{w}$)$_\tau$.
  \begin{align}
    &-\kappa\Delta_0\alpha+\nabla{\gamma}(R(\alpha)\nabla u)\cdot R(\alpha+{\textstyle\frac{\pi}{2}})\nabla u=0,\label{E1}
    \\
    &\frac{u-\overline{w}}{\tau}-\mathrm{div}\Bigl(\hspace{-0.5ex}~^\top\hspace{-0.5ex} R(\alpha)\nabla{\gamma}(R(\alpha)\nabla u)+\nu|\nabla u|^{p-2}\nabla u+\frac{\mu}{\tau}\nabla(u-\overline{w})\Bigr)+\lambda(u-u_{\mathrm{org}})=0.\label{E2}
  \end{align}
  Then, the problem ($E; \alpha,u,\overline{w}$)$_\tau$ has a solution in the following sense:
  \begin{gather*}
    \kappa(\nabla\alpha,\nabla\varphi)_{[L^2(\Omega)]^2}
    +(\nabla{\gamma}(R(\alpha)\nabla u) \cdot 
    R\left(\alpha+{\ts \frac{\pi}{2}}\right)\nabla u,\varphi)_{L^2(\Omega)}
    =0,
    \\
    \mbox{ for all }\varphi\in  H^1_0(\Omega),
  \end{gather*}
  and
  \begin{align*}
    &\frac{1}{\tau}(u-\overline{w},\psi)_{L^2(\Omega)}+(\nabla\gamma(R(\alpha)\nabla u),R(\alpha)\nabla\psi)_{[L^2(\Omega)]^2}+\frac{\mu}{\tau}(\nabla(u-\overline{w}),\nabla\psi)_{[L^2(\Omega)]^2}
    \\
    &\quad +\nu\int_\Omega|\nabla u|^{p-2}\nabla u\cdot\nabla\psi\,dx
    +\lambda(u-u_{\mathrm{org}},\psi)_{L^2(\Omega)}=0,~~~\mbox{ for all }\psi\in W^{1,p}_0(\Omega).
  \end{align*}
  Moreover, if $\kappa$ and $\tau$ satisfy the following largeness conditions:
\begin{align}
  &\kappa>2C_1:=8\sqrt{2}(C_{H^1}^{L^{\frac{2p}{p-2}}}+C_{H^1}^{L^{\frac{2p}{p-1}}})^2(1+C_P)^2|\nabla\gamma|_{W^{1,\infty}(\R^2;\R^{2})}|\nabla u|_{[L^p(\Omega)]^2}^2,~~~~\label{C_1}
  \\
  &\tau<C_2:=\frac{{C_1}(1\land\mu)(1+|\nabla\gamma|_{W^{1,\infty}(\R^2;\R^{2})})^{-2}\left(1+|\nabla u|_{[L^p(\Omega)]^2}\right)^{-2}}{54(1+C_P)^2(1+C_{H^1}^{L^{\frac{2p}{p-2}}})^2(1+2C_1)}.\label{C_2}
\end{align}
  then the solution is unique. 
  \begin{proof}
    According to standard results in the calculus of variations (see, e.g., \cite{MR0390843, MR1727362}), a solution to ($E; \alpha,u,\overline{w}$)$_\tau$ corresponds to a minimizer of the following proper and weakly l.s.c. functional on $[L^2(\Omega)]^2$:
    \begin{equation*}
      \Psi_{\overline{w}}:[y,z]\in [L^2(\Omega)]^2\mapsto\Psi_{\overline{w}}(y,z):=\left\{
      \begin{aligned}
        &E(y,z)+ \frac{1}{2\tau}\int_\Omega |z-\overline{w}|^2\,dx
        \\
        &\qquad+\frac{\mu}{2\tau}\int_\Omega |\nabla(z-\overline{w})|^2\,dx
        \\
        &\qquad \mbox{ if }[y,z]\in H^1_0(\Omega) \times W^{1,p}_0(\Omega),
        \\
        &+\infty, \mbox{ otherwise}.
      \end{aligned}
      \right.
    \end{equation*}
    Moreover, for any $[y,z]\in [L^2(\Omega)]^2$, we have
    \begin{align*}
      \Psi_{\overline{w}}(y,z)\geq \frac{\kappa}{2(C_P)^2}|y|^2_{L^2(\Omega)}+\frac{\lambda}{4}|z|^2_{L^2(\Omega)}-\frac{\lambda}{2}|u_{\mathrm{org}}|^2_{L^2(\Omega)},
    \end{align*}
    where $C_P$ is the Poincar\'e constant. It follows from standard results in the calculus of variations that the existence of a minimizer of $\Psi_{\overline{w}}$ follows easily.

    Next, we set $[\tilde{\alpha},\tilde{u}]\in H^1_0(\Omega) \times W^{1,p}_0(\Omega)$ as the another solution to 
    \\
    ($E; \tilde{\alpha},\tilde{u},\overline{w}$)$_\tau$, and confirm that $[\tilde{\alpha},\tilde{u}]$ coincides with the solution $[\alpha,u]$ to ($E; \alpha,u,\overline{w}$)$_\tau$ obtained via minimizing problem, under the largeness condition of $\kappa$ and $\tau$ as in \eqref{C_1} and \eqref{C_2}. In this case, let us take the difference between ($E;\alpha,u,\overline{w}$)$_\tau$ and ($E;\tilde{\alpha},\tilde{u},\overline{w}$)$_\tau$, multiplying both sides of \eqref{E1} by $(\alpha-\tilde{\alpha})$, multiplying both sides of \eqref{E2} by $(u-\tilde{u})$, and adding the two equations, we obtain:
    \begin{align*}
      &\kappa|\nabla(\alpha-\tilde{\alpha})|^2_{[L^2(\Omega)]^2}+\frac{1}{\tau}\bigl(|u-\tilde{u}|^2_{L^2(\Omega)}+\mu|\nabla(u-\tilde{u})|^2_{[L^2(\Omega)]^2}\bigr)
      \\
      &\quad+\lambda|u-\tilde{u}|^2_{L^2(\Omega)}+\nu\int_\Omega(|\nabla u|^{p-2}\nabla u-|\nabla \tilde{u}|^{p-2}\nabla \tilde{u})\cdot\nabla(u-\tilde{u})\,dx
      \\
      &\quad+\int_\Omega\bigl(\nabla\gamma(R(\alpha)\nabla u)\cdot R(\alpha+{\ts\frac{\pi}{2}})\nabla u-\nabla\gamma(R(\tilde{\alpha})\nabla \tilde{u})\cdot R(\tilde{\alpha}+{\ts\frac{\pi}{2}})\nabla \tilde{u}\bigr)(\alpha-\tilde{\alpha})\,dx
      \\
      &\quad+\int_\Omega\bigl(^\top R(\alpha)\nabla\gamma(R(\alpha)\nabla u)-^\top R(\tilde{\alpha})\nabla\gamma(R(\tilde{\alpha})\nabla \tilde{u})\bigr)\cdot\nabla(u-\tilde{u})\,dx=0.
    \end{align*}
    Also, by using (A2), and Tartar's inequality for p-Laplace operator (cf. \cite[Lemma 2.1]{MR2802983}), we can see that 
    \begin{align}
      &\kappa|\nabla(\alpha-\tilde{\alpha})|^2_{[L^2(\Omega)]^2}+\frac{1}{\tau}\left(|u-\tilde{u}|^2_{L^2(\Omega)}+\mu|\nabla(u-\tilde{u})|^2_{[L^2(\Omega)]^2}\right)\notag
      \\
      &\leq 4\sqrt{2}|\nabla^2\gamma|_{L^\infty(\R^2;\R^{2\times2})}\int_\Omega|\nabla u|^2|\alpha-\tilde{\alpha}|^2\,dx\notag
      \\
      &\quad +8|\nabla^2\gamma|_{L^\infty(\R^2;\R^{2\times2})}\int_\Omega|\nabla(u-\tilde{u})||\nabla u||\alpha-\tilde{\alpha}|\,dx\notag
      \\
      &\quad +4|\nabla\gamma|_{L^\infty(\R^2;\R^{2})}\int_\Omega|\nabla u||\alpha-\tilde{\alpha}|^2\,dx\notag
      \\
      &\quad +6|\nabla\gamma|_{L^\infty(\R^2;\R^{2})}\int_\Omega|\nabla(u-\tilde{u})||\alpha-\tilde{\alpha}|\,dx\notag
      \\
      &\quad +2|\nabla^2\gamma|_{L^\infty(\R^2;\R^{2\times2})}|\nabla(u-\tilde{u})|^2_{[L^2(\Omega)]^2}\notag
      \\[0.5ex]
      &=:I_1+I_2+I_3+I_4+2|\nabla^2\gamma|_{L^\infty(\R^2;\R^{2\times2})}|\nabla(u-\tilde{u})|^2_{[L^2(\Omega)]^2}.\label{lem004}
    \end{align}
    Here, let us denote by $C_{H^1}^{L^{\frac{2p}{p-2}}}$ and $C_{H^1}^{L^{\frac{2p}{p-1}}}$ the constants of two-dimensional Sobolev embeddings $H^1(\Omega)\subset L^{\frac{2p}{p-2}}(\Omega)$ and $H^1(\Omega)\subset L^{\frac{2p}{p-1}}(\Omega)$, respectively. Additionally, let $C_P$ represent the constant related to Poincar\'e's inequality. By applying Young's inequality, the integral terms $I_j$, for $j=1,2,3,4$, in \eqref{lem004} can be estimated as follows:
\begin{align}
  I_1&\leq 4\sqrt{2}|\nabla\gamma|_{W^{1,\infty}(\R^2;\R^{2})}|\nabla u|^2_{[L^p(\Omega)]^2}|\alpha-\tilde{\alpha}|^2_{L^{\frac{2p}{p-2}}(\Omega)}\notag
  \\
  &\leq 4\sqrt{2}(C_{H^1}^{L^{\frac{2p}{p-2}}})^2|\nabla\gamma|_{W^{1,\infty}(\R^2;\R^{2})}|\nabla u|^2_{[L^p(\Omega)]^2}|\alpha-\tilde{\alpha}|^2_{H^1(\Omega)}\notag
  \\
  &\leq 4\sqrt{2}(C_{H^1}^{L^{\frac{2p}{p-2}}})^2(1+C_P)^2|\nabla\gamma|_{W^{1,\infty}(\R^2;\R^{2})}|\nabla u|^2_{[L^p(\Omega)]^2}|\nabla(\alpha-\tilde{\alpha})|^2_{[L^2(\Omega)]^2},\label{I_1}
  \\[1.0ex]
  I_2&\leq 8|\nabla\gamma|_{W^{1,\infty}(\R^2;\R^{2})}|\nabla(u-\tilde{u})|_{[L^2(\Omega)]^2}|\nabla u|_{[L^p(\Omega)]^2}|\alpha-\tilde{\alpha}|_{L^{\frac{2p}{p-2}}(\Omega)}\notag
  \\
  &\leq 8C_{H^1}^{L^{\frac{2p}{p-2}}}|\nabla\gamma|_{W^{1,\infty}(\R^2;\R^{2})}|\nabla(u-\tilde{u})|_{[L^2(\Omega)]^2}|\nabla u|_{[L^p(\Omega)]^2}|\alpha-\tilde{\alpha}|_{H^1(\Omega)}\notag
  \\
  &\leq \frac{64}{\kappa}(C_{H^1}^{L^{\frac{2p}{p-2}}})^2(1+C_P)^2|\nabla\gamma|_{W^{1,\infty}(\R^2;\R^{2})}^2|\nabla u|_{[L^p(\Omega)]^2}|\nabla(u-\tilde{u})|_{[L^2(\Omega)]^2}^2 \notag
  \\
  &\hspace{8ex}+\frac{\kappa}{4}|\nabla(\alpha-\tilde{\alpha})|^2_{[L^2(\Omega)]^2},\label{I_2}
\\
  I_3&\leq 4|\nabla\gamma|_{W^{1,\infty}(\R^2;\R^{2})}|\nabla u|_{[L^p(\Omega)]^2}|\alpha-\tilde{\alpha}|^2_{L^{\frac{2p}{p-1}}(\Omega)}\notag
  \\
  &\leq 4(C_{H^1}^{L^{\frac{2p}{p-1}}})^2|\nabla\gamma|_{W^{1,\infty}(\R^2;\R^{2})}|\nabla u|_{[L^p(\Omega)]^2}|\alpha-\tilde{\alpha}|^2_{H^1(\Omega)}\notag
  \\
  &\leq 4(C_{H^1}^{L^{\frac{2p}{p-1}}})^2(1+C_P)^2|\nabla\gamma|_{W^{1,\infty}(\R^2;\R^{2})}|\nabla u|_{[L^p(\Omega)]^2}|\nabla(\alpha-\tilde{\alpha})|^2_{[L^2(\Omega)]^2},\label{I_3}
  \\[1.0ex]
  I_4&\leq 6|\nabla\gamma|_{W^{1,\infty}(\R^2;\R^{2})}|\nabla(u-\tilde{u})|_{[L^2(\Omega)]^2}|\alpha-\tilde{\alpha}|_{L^2(\Omega)}\notag
  \\
  &\leq \frac{36}{\kappa}(C_P)^2|\nabla\gamma|_{W^{1,\infty}(\R^2;\R^{2})}^2|\nabla(u-\tilde{u})|^2_{[L^2(\Omega)]^2}
  +\frac{\kappa}{4}|\nabla(\alpha-\tilde{\alpha})|^2_{[L^2(\Omega)]^2}.\label{I_4}
\end{align}
Combing \eqref{lem004}--\eqref{I_4}, we arrive at 
\begin{gather*}
  \left(\frac{\kappa}{2}-C_1(u)\right)|\nabla(\alpha_1-\alpha_2)|^2_{[L^2(\Omega)]^2}
  \nonumber
  \\
  +\left(\frac{1}{\tau}-\frac{1}{C_2(u)}\right)\left(|u_1-u_2|^2_{L^2(\Omega)}+\mu|\nabla(u_1-u_2)|^2_{[L^2(\Omega)]^2}\right)\leq0, 
\end{gather*}
where $C_1=C_1(u)$, $C_2=C_2(u)$ are positive constants depending on the minimizer $u$ to $\Psi_{_{\overline{w}}}$, given as:
\begin{align*}
  &C_1:=4\sqrt{2}(C_{H^1}^{L^{\frac{2p}{p-2}}}+C_{H^1}^{L^{\frac{2p}{p-1}}})^2(1+C_P)^2|\nabla\gamma|_{W^{1,\infty}(\R^2;\R^{2})}(1+|\nabla u|_{[L^p(\Omega)]^2})^2,
  \\
  &C_2:=\frac{{C_1}(1\land\mu)(1+|\nabla\gamma|_{W^{1,\infty}(\R^2;\R^{2})})^{-2}\left(1+|\nabla u|_{[L^p(\Omega)]^2}\right)^{-2}}{54(1+C_P)^2(1+C_{H^1}^{L^{\frac{2p}{p-2}}})^2(1+2C_1)}.
\end{align*}
Thus, if $\kappa$ and $\tau$ satisfy the following conditions:
\begin{align}\label{cond1}
  \kappa>2C_1, \mbox{ and } 0<\tau<C_2, 
\end{align}
then the solution to ($E; \alpha,u,\overline{w}$)$_\tau$ is unique. 
  \end{proof}
\end{lemma}
  \begin{proof}[Proof of Theorem \ref{003Thm1}]
    The existence of the solution $\{[\alpha^i, u^i]\}_{i=1}^m$ to (AP)$^\tau$ follows from Lemma \ref{003lemma1} with ${\overline{w}}:=u^{i-1}$. Moreover, the energy-inequality can be observed. In fact, since the solution $[\alpha^i, u^i]$ to $E(\alpha^i,u^i,u^{i-1})$ is the minimizer to $\Psi_i:=\Psi_{u^{i-1}}$, one can compute that
    \begin{align*}
      E(\alpha^i,u^i)&\leq\Psi_i(\alpha^i,u^i)\leq\Psi_{i}(\alpha^{i-1},u^{i-1})=E(\alpha^{i-1},\alpha^{i-1})\leq\Psi_{i-1}(\alpha^{i-1},\alpha^{i-1})
      \\
      &\leq\cdots\leq E(\alpha^1,u^1)\leq\Psi_1(\alpha^1,u^1)\leq\Psi_1(\alpha^0,u_0)= E(\alpha^0,u_0).
    \end{align*}
    This implies the energy-inequality. In particular, we have
    \begin{align}\label{lem003}
      \frac{\nu}{p}|\nabla u^i|^p_{[L^p(\Omega)]^2}\leq E(0,u_0), \mbox{ for any } i=1,2,\dots,m.
    \end{align}

    Now, our remaining task is to show the uniqueness. In this light, we set 
    \begin{equation}\label{kappahat}
      \widehat{\kappa}:=8\sqrt{2}(C_{H^1}^{L^{\frac{2p}{p-2}}}+C_{H^1}^{L^{\frac{2p}{p-1}}})^2(1+C_P)^2|\nabla\gamma|_{W^{1,\infty}(\R^2;\R^{2})}(1+\frac{\nu}{p}E(0,u_0))^\frac{2}{p},
      \end{equation}
      \begin{equation*}
      \widehat{\tau}:=\frac{\widehat{\kappa}(1\land\mu)(1+|\nabla\gamma|_{W^{1,\infty}(\R^2;\R^{2})})^{-2}(1+(\frac{\nu}{p}E(0,u_0))^\frac{1}{p})^{-2}}{54(1+C_P)^2(1+C_{H^1}^{L^{\frac{2p}{p-2}}})^2(1+2\widehat{\kappa})}.
    \end{equation*}
    Then, according to the energy-inequality, we immediately see that when $\kappa > \widehat{\kappa}$ and $\tau < \widehat{\tau}$ satisfies the condition \eqref{cond1}. Therefore, owing to Lemma \ref{003lemma1}, we have the uniqueness, and thus, we complete the proof.
  \end{proof}

  \section{Proof of Main Theorem.}
The proof of Main Theorem is established by verifying several claims, which are divided into the following subsections:\medskip
\\
Subsection \ref{sub4.1} Uniqueness and continuous dependence of system (S).
\\
Subsection \ref{sub4.2} Existence of a limit $[\alpha,u]$ for a subsequence of $\{[\alpha]_\tau,[u]_\tau\}$.
\\
Subsection \ref{sub4.3} Existence and energy-dissipation for the system (S).
\\
Subsection \ref{sub4.4} Verification of Boundedness: $0\leq u \leq 1$ in $\overline{Q}$.
\subsection{Uniqueness and continuous dependence of system (S)}\label{sub4.1}
In this subsection, we show the uniqueness and continuous dependence of the solution to system (S). Let $[\alpha_k ,u_k](k=1,2)$ denote the solutions of the system (S) corresponding to the initial condition $u_1(0)=u_2(0)=u_0$. We consider the difference between the two variational identities for $\alpha_k$, and substitute $\varphi:=(\alpha_1-\alpha_2)(t)$. Using (A2), we can deduce that
\begin{align}
  & \frac{d}{dt}\Big(\kappa\int_0^t|\nabla(\alpha_1-\alpha_2)(t)|^2_{[L^2(\Omega)]^2}dt\Big)
  \nonumber
  \\
  &\leq4\sqrt{2}|\nabla^2\gamma|_{L^\infty(\R^2;\R^{2\times2})}\int_\Omega|\nabla u_1(t)|^2|(\alpha_1-\alpha_2)(t)|^2\,dx
  \nonumber
  \\
  & \quad+2|\nabla^2\gamma|_{L^\infty(\R^2;\R^{2\times2})}\int_\Omega|\nabla u_1(t)||\nabla(u_1-u_2)(t)||(\alpha_1-\alpha_2)(t)|\,dx
  \nonumber
  \\
  &\quad+4|\nabla\gamma|_{L^\infty(\R^2;\R^2)}\int_\Omega|\nabla u_1(t)||(\alpha_1-\alpha_2)(t)|^2\,dx
  \nonumber
  \\
  &\quad+\sqrt{2}|\nabla\gamma|_{L^\infty(\R^2;\R^2)}\int_\Omega|\nabla(u_1-u_2)(t)||(\alpha_1-\alpha_2)(t)|\,dx
  \nonumber
  \\
  &=:J_1+J_2+J_3+J_4.
\label{unique01}
\end{align}
Here, using similar argument as in \eqref{I_1}--\eqref{I_4} with \eqref{lem003}, we can estimate the integral terms $J_k$, $k = 1,2,3,4$, appearing in \eqref{unique01} as follows:
\begin{align}
  J_1 &\leq4\sqrt{2}|\nabla\gamma|_{W^{1,\infty}(\R^2;\R^2)}|\nabla u_1(t)|^2_{[L^p(\Omega)]^2}|(\alpha_1-\alpha_2)(t)|^2_{L^{\frac{2p}{p-2}}(\Omega)}
  \nonumber
  \\
  &\leq4\sqrt{2}\left(\frac{p}{\nu}E(0,u_0)\right)^{\frac{2}{p}}(C_{H^1}^{L^{\frac{2p}{p-2}}})^2(1+C_P)^2|\nabla\gamma|_{W^{1,\infty}(\R^2;\R^2)}|\nabla(\alpha_1-\alpha_2)(t)|^2_{[L^2(\Omega)]^2},
  \label{unique02}
  \\[2ex]
  J_2&\leq2|\nabla\gamma|_{W^{1,\infty}(\R^2;\R^2)}|\nabla u_1(t)|_{[L^p(\Omega)]^2}|\nabla(u_1-u_2)(t)|_{[L^2(\Omega)]^2}|(\alpha_1-\alpha_2)(t)|_{L^{\frac{2p}{p-2}}(\Omega)}
  \nonumber
  \\
    &\leq \frac{8}{\kappa}(C_{H^1}^{L^{\frac{2p}{p-2}}})^2(1+C_P)^2|\nabla\gamma|_{W^{1,\infty}(\R^2;\R^2)}^2|\nabla u_1(t)|_{[L^p(\Omega)]^2}^2|\nabla(u_1-u_2)(t)|_{[L^2(\Omega)]^2}^2 
    \nonumber
    \\
    &\qquad+\frac{\kappa}{8}|\nabla(\alpha_1-\alpha_2)(t)|^2_{[L^2(\Omega)]^2},
    \label{unique03}
\end{align}
\begin{align}
  J_3&\leq4|\nabla\gamma|_{W^{1,\infty}(\R^2;\R^2)}|\nabla u_1(t)|_{[L^p(\Omega)]^2}|(\alpha_1-\alpha_2)(t)|^2_{L^{\frac{2p}{p-1}}(\Omega)}
  \nonumber
  \\
  &\leq 4\left(\frac{p}{\nu}E(0,u_0)\right)^{\frac{1}{p}}(C_{H^1}^{L^{\frac{2p}{p-1}}})^2(1+C_P)^2|\nabla\gamma|_{W^{1,\infty}(\R^2;\R^2)}|\nabla(\alpha_1-\alpha_2)(t)|^2_{[L^2(\Omega)]^2},
  \label{unique04}
\\
    J_4 & \leq \frac{4}{\kappa}(1+C_P)^2|\nabla\gamma|_{W^{1,\infty}(\R^2;\R^2)}^2|\nabla(u_1-u_2)(t)|_{[L^2(\Omega)]^2}^2+\frac{\kappa}{8}|\nabla(\alpha_1-\alpha_2)(t)|^2_{[L^2(\Omega)]^2},
  \label{unique05}
\end{align}
where, $C_{H1}^{L^{\frac{2p}{p-2}}}$ and $C_{H1}^{L^{\frac{2p}{p-1}}}$ denote the constants corresponding to the two-dimensional Sobolev embeddings $H^1(\Omega) \subset L^{\frac{2p}{p-2}}(\Omega)$ and $H^1(\Omega) \subset L^{\frac{2p}{p-1}}(\Omega)$, respectively, and $C_P$ denotes the constant associated with Poincar\'e's inequality.

\noindent
Combining \eqref{unique01}--\eqref{unique05}, and recording the setting of $\widehat{\kappa}$ as in \eqref{kappahat}, we arrive at
\begin{align}
  & \frac{d}{dt}\left(\left(\frac{3\kappa}{4}-\frac{\widehat{\kappa}}{2}\right)\int_0^t|\nabla(\alpha_1-\alpha_2)(t)|^2_{[L^2(\Omega)]^2}dt\right)
  \label{uni001}
  \\
  &~\leq C_3\bigl(1+|\nabla u_1(t)|_{[L^p(\Omega)]^2}\bigr)^2\bigl(|(u_1-u_2)(t)|^2_{L^2(\Omega)}+\mu|\nabla(u_1-u_2)(t)|^2_{[L^2(\Omega)]^2}\bigr), \nonumber
  \\
  &\hspace{30ex}\mbox{ a.e. } t \in (0,T),
  \nonumber
\end{align}
where 
\begin{align*}
    C_3:=\frac{8(1+C_{H^1}^{L^{\frac{2p}{p-2}}})^2(1+C_P)^2|\nabla\gamma|^2_{W^{1,\infty}(\R^2;\R^2)}}{\widehat{\kappa}(1\land\mu)}.
\end{align*}

On the other hand, by putting $ \psi = u_2 $ in the variational inequality for $ u_1 $ and $ \psi = u_1 $ in the one for $ u_2 $, and then summing the two inequalities, we can see that 
\begin{align*}
  &\frac{1}{2}\frac{d}{dt}\bigl(|(u_1-u_2)(t)|^2_{L^2(\Omega)}+\mu|\nabla(u_1-u_2)(t)|^2_{[L^2(\Omega)]^2}\bigr)+\lambda|(u_1-u_2)(t)|^2_{L^2(\Omega)}
  \\
  &+\int_\Omega(|\nabla u_1(t)|^{p-2}\nabla u_1(t)-|\nabla u_2(t)|^{p-2}\nabla u_2(t))\cdot\nabla(u_1-u_2)(t)\,dx 
  \\
  &-\int_\Omega\left[{^\top}R(\alpha_1(t))\nabla\gamma(R(\alpha_1(t))\nabla u_1(t))-{^\top}R(\alpha_2(t))\nabla\gamma(R(\alpha_2(t))\nabla u_2(t))\right]\cdot
  \\
  & \qquad\cdot\nabla(u_1-u_2)(t)\,dx\leq0,\mbox{ a.e. }t\in(0,T).
\end{align*}
Also, by applying (A2), Tartar's inequality for $p$-Laplace operator (cf. \cite[Lemma 2.1]{MR2802983}), and H\"older's inequality, we obtain that
\begin{align}
  & \frac{1}{2}\frac{d}{dt}\bigl(|(u_1-u_2)(t)|^2_{L^2(\Omega)}+\mu|\nabla(u_1-u_2)(t)|^2_{[L^2(\Omega)]^2}\bigr)
  \nonumber
  \\
  &\leq4|\nabla\gamma|_{L^\infty(\R^2;\R^{2})}\int_\Omega|\nabla(u_1-u_2)(t)||(\alpha_1-\alpha_2)(t)|\,dx 
  \nonumber
  \\
  &\quad+4\sqrt{2}|\nabla^2\gamma|_{L^\infty(\R^2;\R^{2\times2})}\int_\Omega|\nabla u_1(t)||\nabla(u_1-u_2)(t)||(\alpha_1-\alpha_2)(t)|\,dx
  \nonumber
  \\
  &\quad+2|\nabla^2\gamma|_{L^\infty(\R^2;\R^{2\times2})}|\nabla(u_1-u_2)(t)|^2_{[L^2(\Omega)]^2}
  \nonumber
  \\
  &=: J_5+J_6+2|\nabla^2\gamma|_{L^\infty(\R^2;\R^{2\times2})}|\nabla(u_1-u_2)(t)|^2_{[L^2(\Omega)]^2}.
  \label{unique06}
\end{align}

Here, by using the two-dimensional Sobolev embedding $H^1(\Omega) \subset L^{\frac{2p}{p-2}}(\Omega)$, Poincar\'e's inequality and Young's inequality, the integral terms $J_5$ and $J_6$ in \eqref{unique06} can be calculated as follows:
\begin{align}
  J_5&\leq 4|\nabla\gamma|_{W^{1,\infty}(\R^2;\R^2)}|\nabla(u_1-u_2)(t)|_{[L^2(\Omega)]^2}|(\alpha_1-\alpha_2)(t)|_{L^2(\Omega)},\nonumber
  \\
  &\leq \frac{32}{\kappa}(1+C_P)^2|\nabla\gamma|_{W^{1,\infty}(\R^2;\R^2)}^2|\nabla(u_1-u_2)(t)|_{[L^2(\Omega)]^2}^2+\frac{\kappa}{8}|\nabla(\alpha_1-\alpha_2)(t)|^2_{[L^2(\Omega)]^2},
  \label{unique07}
\\
  J_6&\leq4\sqrt{2}|\nabla\gamma|_{W^{1,\infty}(\R^2;\R^2)}|\nabla u_1(t)|_{[L^p(\Omega)]^2}|\nabla(u_1-u_2)(t)|_{[L^2(\Omega)]^2}|(\alpha_1-\alpha_2)(t)|_{L^{\frac{2p}{p-2}}(\Omega)}
  \nonumber
  \\
    &\leq \frac{64}{\kappa}({C_{H^1}^{L^{\frac{2p}{p-2}}}})^2(1+C_P)^2|\nabla\gamma|^2_{W^{1,\infty}(\R^2;\R^2)}|\nabla u_1(t)|_{[L^p(\Omega)]^2}^2|\nabla(u_1-u_2)(t)|_{[L^2(\Omega)]^2}^2 
    \nonumber
    \\
    &\qquad\qquad +\frac{\kappa}{8}|\nabla(\alpha_1-\alpha_2)(t)|^2_{[L^2(\Omega)]^2},
     \label{unique08}
\end{align}
where, $C_{H^1}^{L^{\frac{2p}{p-2}}}$ and $C_{H^1}^{L^{\frac{2p}{p-1}}}$ stand for the embedding constants of $H^1(\Omega) \subset L^{\frac{2p}{p-2}}(\Omega)$ and $H^1(\Omega) \subset L^{\frac{2p}{p-1}}(\Omega)$, respectively, while $C_P$ represents the constant associated with Poincar\'e's inequality.

\noindent
Combining \eqref{unique06}--\eqref{unique08}, we arrive at 
\begin{align}
  &\frac{d}{dt}\bigl(|(u_1-u_2)(t)|^2_{L^2(\Omega)}+\mu|\nabla(u_1-u_2)(t)|^2_{[L^2(\Omega)]^2}\bigr)
  \nonumber
  \\
  &\quad\leq C_4\bigl(1+|\nabla u_1(t)|_{[L^p(\Omega)]^2}\bigr)^2\bigl(|(u_1-u_2)(t)|^2_{L^2(\Omega)}+\mu|\nabla(u_1-u_2)(t)|^2_{[L^2(\Omega)]^2}\bigr)
  \nonumber
  \\
   &\qquad\qquad\quad+\frac{\kappa}{4}|\nabla(\alpha_1-\alpha_2)(t)|^2_{[L^2(\Omega)]^2},\mbox{ a.e. }t\in(0,T),\label{uni002}
\end{align}where
\begin{gather*}
  C_4:=\frac{64\bigl(1+C_{H^1}^{L^{\frac{2p}{p-2}}}\bigr)^2(1+C_P)^2(1+|\nabla\gamma|_{W^{1,\infty}(\R^2;\R^{2})})^2(1+\widehat{\kappa})}{\widehat{\kappa}(1\land\mu)}.
\end{gather*}
Now, we set
\begin{align*}
  \kappa_*:=\widehat{\kappa}, \mbox{ and } C_*:=C_3+C_4.
\end{align*}
Then, by adding inequalities \eqref{uni001} and \eqref{uni002}, we arrive at the desired inequality \eqref{cncl_J}.
\subsection{Existence of a limit $[\alpha,u]$ for a subsequence of $\{[\alpha]_\tau,[u]_\tau\}$}\label{sub4.2}
First, taking into account Notation \ref{deftseq} and Example \ref{ex1}, we observe that
\begin{gather}
  \nabla\gamma(R([\overline{\alpha}]_\tau(t))\nabla [\overline{u}]_\tau(t))\in\partial\Phi(R( [\overline{\alpha}]_\tau(t))\nabla [\overline{u}]_\tau(t))\mbox{ in }[L^2(\Omega)]^2,
  \label{subdig1}
  \\
  \mbox{ for a.e. }t\in(0,T),
  \nonumber
\end{gather}
and hence, 
\begin{gather}
  \nabla\gamma(R([\overline{\alpha}]_\tau)\nabla [\overline{u}]_\tau)\in\partial\widehat{\Phi}^I(R( [\overline{\alpha}]_\tau)\nabla [\overline{u}]_\tau)\mbox{ in }L^2(0,T;[L^2(\Omega)]^2),
  \label{subdig2}
  \\
  \mbox{ any open interval }I\subset(0,T),
  \nonumber
\end{gather}
where  $[{u}]_\tau$, $ [\overline{u}]_\tau $, and $ [\underline{u}]_\tau $  are the time-interpolation of $ \{ u^i \}_{i = 1}^m $ as in Notation \ref{deftseq}, and $[{\alpha}]_\tau$, $ [\overline{\alpha}]_\tau $, and $ [\underline{\alpha}]_\tau $ are those of $ \{ \alpha^i\}_{i = 1}^m $,  for $ \tau \in (0, \widehat{\tau}) $.
\medskip

From Theorem \ref{003Thm1}, since $\kappa$ is greater than $\widehat{\kappa}$, we obtain the following boundedness:
\begin{enumerate}
  \item[{\bf(B-1)}] $\{[\overline{\alpha}]_{\tau}~|~\tau\in(0,\widehat{\tau})\}$ and $\{[\underline{\alpha}]_{\tau}~|~\tau\in(0,\widehat{\tau})\}$ are bounded in $L^\infty(0,T;H^1_0(\Omega))$,
  \item[{\bf(B-2)}] $\{[u]_{\tau}~|~\tau\in(0,\widehat{\tau})\}$ is bounded in $L^\infty(0,T;W^{1,p}_0(\Omega))$ and in $W^{1,2}(0,T;H^1_0(\Omega))$,
  \item[{\bf(B-3)}] $\{[\overline{u}]_{\tau}~|~\tau\in(0,\widehat{\tau})\}$ and $\{[\underline{u}]_{\tau}~|~\tau\in(0,\widehat{\tau})\}$ are bounded in $L^\infty(0,T;W^{1,p}_0(\Omega))$,
  \item[{\bf(B-4)}] $\{\nabla\gamma(R([\overline{\alpha}]_\tau)\nabla [\overline{u}]_\tau)~|~\tau\in(0,\widehat{\tau})\}$ is bounded in $L^\infty(Q;\R^2)$,
  \item[{\bf(B-5)}] The function of time $t\in[0,T]\mapsto E([\overline{\alpha}]_{\tau}(t),[\overline{u}]_{\tau}(t))\in[0,\infty)$ and $t\in [0,T]\mapsto E([\underline{\alpha}]_{\tau}(t),[\underline{u}]_{\tau}(t))\in[0,\infty)$ are nonincreasing for every $0<\tau<\widehat{\tau}$. Moreover, $E({\alpha}^{0},u_0)$ is bounded, and hence, the class $\{E([\overline{u}]_{\tau},[\overline{\alpha}]_{\tau}),$ $E([\underline{u}]_{\tau},[\underline{\alpha}]_{\tau})~|~\tau\in(0,\widehat{\tau})\} $ is bounded in $BV(0,T)$.
\end{enumerate}
However, to handle the convergences of terms arisen by the structure of composition function of $\alpha$, we need to derive convergences of $\alpha$ in strong topology.

In view of this, we consider deriving additional boundedness of $\alpha$. Here, by setting $\varphi:=\frac{1}{\tau}(\alpha^i-\alpha^{i-1})$ in (AP1;$\alpha^i$,$u^i$)$-$(AP1;$\alpha^{i-1}$,$u^{i-1}$), we can see that 
\begin{align}
  \frac{\kappa}{\tau}|\nabla(\alpha^i-\alpha^{i-1})|&\leq\frac{4\sqrt{2}}{\tau}|\nabla^2\gamma|_{L^\infty(\R^2;\R^{2\times2})}\int_\Omega |\nabla u^i|^2|\alpha^i-\alpha^{i-1}|^2\,dx\nonumber
  \\
  &\quad+\frac{2}{\tau}|\nabla^2\gamma|_{L^\infty(\R^2;\R^{2\times2})}\int_\Omega |\nabla(u^i-u^{i-1})||\nabla u^i||\alpha^i-\alpha^{i-1}|\,dx\nonumber
  \\
  &\quad+\frac{4}{\tau}|\nabla\gamma|_{L^\infty(\R^2;\R^2)}\int_\Omega |\nabla u^i||\alpha^i-\alpha^{i-1}|^2\,dx\nonumber
  \\
  &\quad+\frac{\sqrt{2}}{\tau}|\nabla\gamma|_{L^\infty(\R^2;\R^2)}\int_\Omega |\nabla(u^i-u^{i-1})||\alpha^i-\alpha^{i-1}|\,dx\nonumber
  \\
  &=:J_7+J_8+J_9+J_{10}.\label{difference_of_alpha}
\end{align}
Here, by arguments similar to those used in \eqref{I_1}--\eqref{I_4} and applying \eqref{lem003}, the integral terms $J_k$, $k=7,8,9,10$, appearing in \eqref{difference_of_alpha} can be estimated as follows:
\begin{align}
  J_7&\leq\frac{4\sqrt{2}}{\tau}|\nabla\gamma|_{W^{1,\infty}(\R^2;\R^{2})}|\nabla u^i|^2_{[L^p(\Omega)]^2}|\alpha^i-\alpha^{i-1}|_{L^{\frac{2p}{p-2}}(\Omega)}^2\notag
  \\
  &\leq \frac{4\sqrt{2}}{\tau}\left(\frac{p}{\nu}E(0,u_0)\right)^{\frac{2}{p}}(C_{H^1}^{L^{\frac{2p}{p-2}}})^2(1+C_P)^2|\nabla\gamma|_{W^{1,\infty}(\R^2;\R^{2})}|\nabla(\alpha^i-\alpha^{i-1})|^2_{[L^2(\Omega)]^2},\label{J_7}
  \\[1.0ex]
    J_8&\leq\frac{2}{\tau}|\nabla\gamma|_{W^{1,\infty}(\R^2;\R^{2})}|\nabla u^i|_{[L^p(\Omega)]^2}|\nabla(u^i-u^{i-1})|_{[L^2(\Omega)]^2}|\alpha^i-\alpha^{i-1}|_{L^{\frac{2p}{p-2}}(\Omega)}\notag
  \\
  &\leq \frac{4}{\kappa\tau}\left(\frac{p}{\nu}E(0,u_0)\right)^\frac{2}{p}(C_{H^1}^{L^{\frac{2p}{p-2}}})^2(1+C_P)^2|\nabla\gamma|_{W^{1,\infty}(\R^2;\R^{2})}^2|\nabla(u^{i}-u^{i-1})|_{[L^2(\Omega)]^2}^2 \notag
  \\
  &\qquad +\frac{\kappa}{4\tau}|\nabla(\alpha^i-\alpha^{i-1})|^2_{[L^2(\Omega)]^2},\label{J_8}
\\
  J_9&\leq\frac{4}{\tau}|\nabla\gamma|_{W^{1,\infty}(\R^2;\R^{2})}|\nabla u^i|_{[L^p(\Omega)]^2}|\alpha^i-\alpha^{i-1}|^2_{L^{\frac{2p}{p-1}}(\Omega)}\notag
  \\
  &\leq \frac{4}{\tau}\left(\frac{p}{\nu}E(0,u_0)\right)^{\frac{1}{p}}(C_{H^1}^{L^{\frac{2p}{p-1}}})^2(1+C_P)^2|\nabla\gamma|_{W^{1,\infty}(\R^2;\R^{2})}|\nabla(\alpha^i-\alpha^{i-1})|^2_{[L^2(\Omega)]^2},\label{J_9}
  \end{align}
\begin{align}
  J_{10}&\leq\frac{\sqrt{2}}{\tau}|\nabla\gamma|_{W^{1,\infty}(\R^2;\R^{2})}\int_\Omega |\nabla(u^i-u^{i-1})||\alpha^i-\alpha^{i-1}|\,dx\notag
  \\
  &\leq \frac{2}{\kappa\tau}(1+C_P)^2|\nabla\gamma|_{W^{1,\infty}(\R^2;\R^{2})}^2|\nabla(u^{i}-u^{i-1})|^2_{[L^2(\Omega)]^2}
  +\frac{\kappa}{4\tau}|\nabla(\alpha^i-\alpha^{i-1})|^2_{[L^2(\Omega)]^2}.\label{J_10}
\end{align}
where, $C_{H^1}^{L^{\frac{2p}{p-2}}}$ and $C_{H^1}^{L^{\frac{2p}{p-1}}}$ are the embedding constants associated with $H^1(\Omega) \subset L^{\frac{2p}{p-2}}(\Omega)$ and $H^1(\Omega) \subset L^{\frac{2p}{p-1}}(\Omega)$, respectively, and $C_P$ is the Poincar\'e's inequality constant.

\noindent
Here, we set:
\begin{align*}
  C_5:=\frac{4\left(1+(\frac{p}{\nu}E(0,u_0))^\frac{1}{p}\right)^2(1+C_P)^2(1+C_{H^1}^{L^{\frac{2p}{p-2}}})^2|\nabla\gamma|_{W^{1,\infty}(\R^2;\R^{2})}^2}{\widehat{\kappa}\mu}.
\end{align*}
Then, by combing \eqref{difference_of_alpha}--\eqref{J_10} and using \eqref{f-ene0}, we arrive at 
\begin{align*}
  \frac{1}{\tau}\left(\frac{\kappa}{2}-\frac{\widehat{\kappa}}{2}\right)|\nabla(\alpha^i-\alpha^{i-1})|^2_{[L^2(\Omega)]^2}&\leq C_5\frac{\mu}{\tau}|\nabla(u^i-u^{i-1})|_{[L^2(\Omega)]^2},
\end{align*}
and hence, 
\begin{align*}
  \frac{\kappa-\widehat{\kappa}}{2\tau}\sum_{i=1}^k|\nabla(\alpha^i-\alpha^{i-1})|^2_{[L^2(\Omega)]^2}&\leq C_5\frac{\mu}{\tau}\sum_{i=1}^k|\nabla(u^i-u^{i-1})|_{[L^2(\Omega)]^2}
  \\
  &\leq C_5 E(0,u_0), 
  \\
  \mbox{ for any } k = 1,\dots,m, &\mbox{ for any } \tau\in(0,\widehat{\tau}).
\end{align*}
Therefore, since $\kappa$ is greater than $\widehat{\kappa}$, we derive the following additional boundedness:
\begin{description}
  \item[(B-6)]$\{[\alpha]_{\tau}~|~\tau\in(0,\widehat{\tau})\}$ is bounded in $W^{1,2}(0,T;H^1_0(\Omega))$.
\end{description}
By virtue of (B-1)--(B-4), (B-6) and 2D-embedding $ W^{1, p}(\Omega) \subset C(\overline{\Omega}) $ with $ p > 2 $, we can apply the compactness theory of Aubin's type \cite[Corollary 4]{MR0916688}, and obtain a sequence $ \{ \tau_n \}_{n\in\N} \subset (0, \widehat{\tau});\tau_n\downarrow0 $, and a pair of functions $[{\alpha},u]\in[L^2(0,T;L^2(\Omega))]^2$ with $\bm{w}^*\in L^\infty(Q;\R^2)$ such that: 
\begin{align}
  &\alpha_n:=[\alpha]_{\tau_n}\rightarrow\alpha\mbox{ in }C([0,T];L^2(\Omega)),\mbox{ weakly in }W^{1,2}(0,T;H^1_0(\Omega)),
  \nonumber
  \\
  &\hspace*{10ex}\mbox{ weakly-}*\mbox{in }L^\infty(0,T;H^1_0(\Omega)), \mbox{ as $ n \to \infty $,}
  \label{conv_1}
  \\[1ex]
  & u_n : =  [u]_{\tau_n}\rightarrow u\mbox{ in }C(\overline{Q}),\mbox{ weakly in } W^{ 1 , 2 }( 0,T ; H^1_0 ( \Omega)),
  \nonumber 
  \\
  &\hspace*{10ex}\mbox{ weakly-}* \mbox{ in } L^\infty( 0,T ; W^{ 1 , p }_0 ( \Omega ) ) , \mbox{ as $ n \to \infty $,} 
  \label{conv_2}
\end{align}
and
\begin{align}
    &\nabla\gamma(R([\overline{\alpha}]_{\tau_n})\nabla [\overline{u}]_{\tau_n})~\rightarrow\bm{w}^*\mbox{ weakly-}*\mbox{in }L^\infty(Q;\R^2), \mbox{ as $ n \to \infty $,}
  \label{conv_3}
\end{align}
in particular,
\begin{align}\label{conv_4}
  u(0)=\lim_{n\rightarrow\infty}u_n(0)=u_0\mbox{ in }[L^2(\Omega)]^2.
\end{align}
Here, since
\begin{align*}
  &\left\{
  \begin{aligned}
    &\max\{|([\overline{\alpha}]_{\tau_n}-\alpha_{n})(t)|_{H^1(\Omega)},|([\underline{\alpha}]_{\tau_n}-\alpha_{n})(t)|_{H^1(\Omega)}\}
    \\
    &\quad\leq \int_{t_{i-1}}^{t_i}|\partial_t\alpha_{n}(t)|_{H^1(\Omega)}\,dt\leq \tau_n^{\frac{1}{2}}|\partial_t\alpha_{n}|_{L^2(0,T;H^1(\Omega))},
    \\
    &\max\{|([\overline{u}]_{\tau_n}-u_{n})(t)|_{H^1(\Omega)},|([\underline{u}]_{\tau_n}-u_{n})(t)|_{H^1(\Omega)}\}
    \\
    &\quad\leq \int_{t_{i-1}}^{t_i}|\partial_tu_{n}(t)|_{H^1(\Omega)}\,dt\leq \tau_n^{\frac{1}{2}}|\partial_tu_{n}|_{L^2(0,T;H^1(\Omega))}, 
  \end{aligned}
  \right.
  \\
  &\qquad \mbox{ for }t\in[t_{i-1},t_i),~i=1,\dots,m, \mbox{ and }\tau\in(0,\widehat{\tau}),
\end{align*}
one can also see from \eqref{conv_1} and \eqref{conv_2} that:
\begin{align}
  &\overline{\alpha}_{n}:=[\overline{\alpha}]_{\tau_n}\rightarrow\alpha,~\underline{\alpha}_{n}:=[\underline{\alpha}]_{\tau_n}\rightarrow\alpha\mbox{ in }L^\infty ( 0,T ; L^2 ( \Omega ) ),
  \nonumber
  \\
  &\hspace*{10ex}\mbox{ weakly-}*\mbox{ in }L^\infty(0,T;H^1_0(\Omega)),
  \label{conv_5}
  \\[1ex]
  &\overline{u}_{n}:=[\overline{u}]_{\tau_n}\rightarrow u,~\underline{u}_{n}:=[\underline{u}]_{\tau_n}\rightarrow u \mbox{ in }L^\infty(0,T;L^2(\Omega)),
  \nonumber
  \\
  &\hspace*{10ex}\mbox{ weakly-}* \mbox{ in }L^\infty(0,T;W^{1,p}_0(\Omega)),
  \label{conv_6}
\end{align}
and in particular,
\begin{align}\label{conv_7}
&\left\{
  \begin{aligned}
    &\overline{\alpha}_{n}(t)\rightarrow \alpha(t), \underline{\alpha}_{n}\rightarrow \alpha(t)\mbox{ in }L^2(\Omega) \mbox{ and weakly in }H^1_0(\Omega),
    \\
    &\overline{u}_{n}(t)\rightarrow u(t), \underline{u}_{n}(t)\rightarrow u(t)\mbox{ in }L^p(\Omega) \mbox{ and weakly in }W^{1,p}_0(\Omega),
  \end{aligned}\right.
  \\
  &\hspace{15ex}\mbox{ as } n\rightarrow\infty,\mbox{ for any }t\in(0,T).
  \nonumber
\end{align}
Moreover, (B-5) enable us to see
\begin{gather}\label{conv_8}
    \bigl| E(\overline{\alpha}_n,\overline{u}_n) -E(\underline{\alpha}_n,\underline{u}_n) \bigr|_{L^1(0, T)} \leq 2E(\alpha^0,u_0) \tau_n \to 0, \mbox{ as $ n \to \infty $}.
\end{gather}
So, applying Helly's selection theorem \cite[Chapter 7, p.167]{rudin1976principles}, we will find a bounded and nonincreasing function $\mathcal{J}_*:[0,T]\mapsto[0,\infty)$, such that 
\begin{equation}\label{conv_9}
    \left\{
    \begin{aligned}
    &E(\overline{\alpha}_{n},\overline{u}_{n})\rightarrow\mathcal{J}_* \mbox{ and } E(\underline{\alpha}_{n},\underline{u}_{n})\rightarrow\mathcal{J}_* 
  \\
  & \qquad \mbox{ weakly-}*\mbox{ in }BV(0,T),\mbox{ and }\mbox{weakly-}* \mbox{ in }L^\infty(0,T),
  \\
    &E (\overline{\alpha}_{n}(t),\overline{u}_{n}(t)) \rightarrow \mathcal{J}_*(t) \mbox{ and } E (\underline{\alpha}_{n}(t),\underline{u}_{n}(t)) \rightarrow \mathcal{J}_*(t), \mbox{ for any }t\in[0,T], 
    \end{aligned}
    \right.
\end{equation}
as $ n \to \infty $, by taking a subsequence if necessary.

Thus, the assertion in subsection \ref{sub4.2} has been established.
\begin{rem}
  If (B-6) is not derived, then the limits of the sequences $\{\overline{\alpha}_n\}_{n\in\N}$ and $\{\underline{\alpha}_n\}_{n\in\N}$ in \eqref{conv_5} do not coincide.
\end{rem}
\subsection{Existence and energy-dissipation for the system (S)}\label{sub4.3}
Now, let us prove that the pair of functions $[\alpha,u]$ is a solution to the system (S). Condition (S3) can be verified by using \eqref{conv_4}. Next, we proceed to show that $[\alpha,u]$ satisfies the variational inequalities (S1) and (S2). For any $ t \in [0, T] $, it follows from (AP1;$\alpha^i,u^i$) and (AP2;$\alpha^i,u^i$) that the sequences given in \eqref{conv_1}--\eqref{conv_6} satisfy the following inequalities:
\begin{align}
  &\int_0^t\int_\Omega\nabla\gamma(R(\overline{\alpha}_{n}(\sigma))\nabla\overline{u}_{n}(\sigma))\cdot R\left(\overline{\alpha}_{n}(\sigma)+{\ts\frac{\pi}{2}}\right)\nabla \overline{u}_{n}(\sigma)\omega(\sigma)\,dxd\sigma
  \nonumber 
  \\ 
  &\quad+\kappa\int_{0}^{t}(\nabla\overline{\alpha}_{n}(\sigma),\nabla\omega(\sigma))_{[L^2(\Omega)]^2}\,d\sigma=0,~~\mbox{ for any }\omega\in L^2(0,T;H^1_0(\Omega)),
  \label{conv_10}
\end{align}
and, 
\begin{align}
  &\int_0^t (\partial_tu_n(\sigma),\overline{u}_n(\sigma)-w(\sigma))_{L^2(\Omega)}\,d\sigma\nonumber
+\mu\int_0^t(\nabla\partial_tu_n(\sigma),\nabla(\overline{u}_n-w)(\sigma))_{[L^2(\Omega)]^2}\,d\sigma
  \nonumber
  \\
  &\quad+{\nu}\int_0^t\int_{\Omega}|\nabla \overline{u}_n(\sigma)|^{p-2}\nabla \overline{u}_n(\sigma)\cdot\nabla(\overline{u}_{n}-w)(\sigma)\,dxd\sigma
  \nonumber
  \\
  &\quad+\lambda\int_0^t(\overline{u}_{n}(\sigma)-u_{\mathrm{org}},(\overline{u}_{n}-w)(\sigma))_{L^2(\Omega)}\,d\sigma
  \nonumber
  \\
  &\quad+\int_0^t\int_\Omega\gamma(R(\overline{\alpha}_{n}(\sigma))\nabla\overline{u}_{n} (\sigma))\,dxd\sigma\leq\int_0^t\int_\Omega\gamma(R(\overline{\alpha}_{n}(\sigma))\nabla w(\sigma))\,dxd\sigma,
  \nonumber
  \\
  &\hspace{25ex}\mbox{ for any }w\in L^2(0,T;W^{1,p}_0(\Omega)).\label{conv_11}
\end{align}
Here, it follows from \eqref{conv_1} and \eqref{conv_2} that
\begin{align}
  &\liminf_{n\rightarrow\infty}\int_0^t|\nabla\overline{u}_n(\sigma)|_{[L^p(\Omega)]^2}^p\,d\sigma\geq\int_0^t|\nabla{u}(\sigma)|_{[L^p(\Omega)]^2}^p\,d\sigma,
  \label{conv_12}
  \\
  &\liminf_{n\rightarrow\infty}\int_0^t|\nabla\overline{\alpha}_n(\sigma)|_{[L^2(\Omega)]^2}^2\,d\sigma\geq\int_0^t|\nabla{\alpha}(\sigma)|_{[L^2(\Omega)]^2}^2\,d\sigma.
  \label{conv_13}
\end{align}
Also, by using \eqref{conv_1} and \eqref{conv_2}, weakly lower-semicontinuity of $\gamma$ and Fatou's lemma, one can see that 
\begin{align}
\liminf_{n\rightarrow\infty}\int_{0}^{t}\int_\Omega\gamma(R(\overline{\alpha}_n(\sigma)\nabla \overline{u}_n(\sigma)))\,dxd\sigma
  \geq \int_0^t\int_\Omega\gamma(R({\alpha}(\sigma))\nabla{u}(\sigma))\,dxd\sigma.
\label{conv_14} 
\end{align}
Furthermore, by setting $m_t:=\bigl([\frac{t}{\tau}]+1\bigr)\land \frac{T}{\tau}$, where $ [\frac{t}{\tau}] \in \mathbb{Z} $ denotes the integer part of $ \frac{t}{\tau} \in \R $, we observe that
\begin{align*}
  &\int_0^t(\nabla\partial_tu_n(\sigma),\nabla\overline{u}_n(\sigma))_{[L^2(\Omega)]^2}\,d\sigma
  \\
  &=\int_{0}^{t_{{m}_t}}(\nabla\partial_tu_n(\sigma),\nabla\overline{u}_n(\sigma))_{[L^2(\Omega)]^2}\,d\sigma-\int_{t}^{t_{m_t}}(\nabla\partial_tu_n(\sigma),\nabla\overline{u}_n(\sigma))_{[L^2(\Omega)]^2}\,d\sigma
  \\
  &\geq\sum^{m_t}_{i=1}\frac{1}{2}(|\nabla u_n(t_i)|_{[L^2(\Omega)]^2}^2-|\nabla u_n(t_{i-1})|_{ [L^2(\Omega)]^2}^2)
  \\
  &\qquad-|\Omega|^{\frac{p-2}{2p}}|\nabla\overline{u}_n|_{L^\infty(0,T;[L^p(\Omega)]^2)}\int_{t}^{t_{m_t}}|\nabla\partial_tu_n(\sigma)|_{[L^2(\Omega)]^2}\,d\sigma
\end{align*}
\begin{align*}
  &\geq\frac{1}{2}(|\nabla u_n(t_{m_t})|_{[L^2(\Omega)]^2}^2-|\nabla u_n(0)|_{ [L^2(\Omega)]^2}^2)
  \\
  &\quad-\tau_n^\frac{1}{2}|\Omega|^{\frac{p-2}{2p}}|\nabla\overline{u}_n|_{L^\infty(0,T;[L^p(\Omega)]^2)}|\nabla\partial_tu_n|_{L^2(0,T;[L^2(\Omega)]^2)}
  \\
  &=\frac{1}{2}(|\nabla\overline{u}_n(t)|^2_{[L^2(\Omega)]^2}-|\nabla u_0|_{[L^2(\Omega)]^2}^2)
  \\
  & \quad -\tau_n^\frac{1}{2}|\Omega|^{\frac{p-2}{2p}}|\nabla\overline{u}_n|_{L^\infty(0,T;[L^p(\Omega)]^2)}|\nabla\partial_tu_n|_{L^2(0,T;[L^2(\Omega)]^2)},\mbox{ for $ n = 1, 2, 3, \dots $.}
\end{align*}
Hence, by \eqref{conv_7}, we obtain
\begin{align}
  &\liminf_{n\rightarrow\infty}\int_0^t(\nabla\partial_tu_n(\sigma),\nabla\overline{u}_n(\sigma))_{[L^2(\Omega)]^2}\,d\sigma
  \label{conv_15}
  \\
  &\quad\geq \frac{1}{2}(|\nabla{u}(t)|^2_{[L^2(\Omega)]^2}-|\nabla u_0|_{[L^2(\Omega)]^2}^2)=\int_{0}^{t}(\nabla\partial_tu(\sigma),\nabla u(\sigma))_{[L^2(\Omega)]^2}\,d\sigma.
  \nonumber
\end{align}
So, by putting $w=u$ in \eqref{conv_11} and using \eqref{conv_2}, \eqref{conv_5} and \eqref{conv_6}, we obtain that
\begin{align}
  &\limsup_{n\rightarrow\infty} 
  \biggl( 
    \frac{\nu}{p}\int_0^t|\nabla\overline{u}_n(\sigma)|_{[L^p(\Omega)]^2}^p\,d\sigma+\int_0^t\int_\Omega\gamma(R( \overline{\alpha}_n(\sigma))\nabla\overline{u}_n(\sigma))\,dxd\sigma
    \nonumber
    \\
   & \quad +\frac{\mu}{2}(|\nabla\overline{u}_n(t)|_{[L^2(\Omega)]^2}^2-|\nabla u_0|_{[L^2(\Omega)]^2}^2)
   \\
   &\qquad\qquad-\tau_n^\frac{1}{2}|\Omega|^\frac{p-2}{2p}|\nabla\overline{u}_n|_{L^\infty(0,T;[L^p(\Omega)]^2)}|\nabla\partial_tu_n|_{L^2(0,T;[L^2(\Omega)]^2)}
  \biggr)
  \nonumber
  \\
  &\leq \limsup_{n\rightarrow\infty} 
  \biggl( 
    \frac{\nu}{p}\int_0^t|\nabla\overline{u}_n(\sigma)|_{[L^p(\Omega)]^2}^p\,d\sigma+\int_0^t\int_\Omega\gamma(R( \overline{\alpha}_n(\sigma))\nabla\overline{u}_n(\sigma))\,dxd\sigma
    \nonumber
    \\
   & \hspace{15ex}
      +\mu\int_0^t(\nabla\partial_tu_n(\sigma),\nabla\overline{u}_n(\sigma))_{[L^2(\Omega)]^2}\,d\sigma
  \biggr)\label{conv_16}
  \\
  & \leq - \lim_{n\rightarrow\infty}\int_0^t (\partial_tu_n(\sigma)+\lambda(\overline{u}_n(\sigma)-u_{\mathrm{org}}),(\overline{u}_n-u)(\sigma))_{L^2(\Omega)}\,d\sigma 
  \nonumber
  \\
  &\quad+\frac{\nu}{p}\int_0^t|\nabla{u}(\sigma)|_{[L^p( \Omega )]^2}^p\,d\sigma+\lim_{n\rightarrow\infty}\int_0^t\int_\Omega\gamma(R(\overline{\alpha}_n(\sigma))\nabla u(\sigma))\,dxd\sigma
  \nonumber
  \\
  &\quad+\lim_{n\rightarrow\infty}\mu\int_0^t(\nabla\partial_tu_n(\sigma),\nabla u(\sigma))_{[L^2(\Omega)]^2}\,d\sigma
  \nonumber
  \\
  &=\frac{\nu}{p}\int_0^t|\nabla{u}(\sigma)|_{[L^p( \Omega )]^2}^p\,d\sigma+\int_{0}^t\int_\Omega\gamma(R({\alpha}(\sigma))\nabla{u}(\sigma))\,dxd\sigma
  \nonumber
  \\
  &\quad + \mu\int_0^t(\nabla\partial_tu(\sigma),\nabla u(\sigma))_{[L^2(\Omega)]^2}\,d\sigma
  \nonumber
  \\
  &=\frac{\nu}{p}\int_0^t|\nabla{u}(\sigma)|_{[L^p( \Omega )]^2}^p\,d\sigma+\int_{0}^t\int_\Omega\gamma(R({\alpha}(\sigma))\nabla{u}(\sigma))\,dxd\sigma
  \nonumber
  \\
  &\quad +\frac{\mu}{2}(|\nabla{u}(t)|^2_{[L^2(\Omega)]^2}-|\nabla u_0|_{[L^2(\Omega)]^2}^2).
  \nonumber
 \end{align}
 Therefore, from (Fact 1) in Section 1, \eqref{conv_12}, \eqref{conv_14}, \eqref{conv_15} and \eqref{conv_16}, we can derive the following convergences as $n\rightarrow\infty$:
 \begin{gather}
  \int_0^t|\nabla\overline{u}_n(\sigma)|_{[L^p(\Omega)]^2}^p\,d\sigma\rightarrow\int_0^t|\nabla{u}(\sigma)|_{[L^p(\Omega)]^2}^p\,d\sigma,
  \label{conv_17}
\\
\int_0^t(\nabla\partial_tu_n(\sigma),\nabla\overline{u}_n(\sigma))_{[L^2(\Omega)]^2}\,d\sigma\rightarrow\int_0^t(\nabla\partial_tu(\sigma),\nabla{u}(\sigma))_{[L^2(\Omega)]^2}\,d\sigma,~~~~~
\label{conv_18}
\\
|\nabla\overline{u}_n(t)|_{[L^2(\Omega)]^2}^2\rightarrow|\nabla{u}(t)|_{[L^2(\Omega)]^2}^2, \mbox{ for any $ t \in [0, T] $.}
  \nonumber
\end{gather} 
Moreover, due to (A2), \eqref{f-ene0}, (B-6), and uniform convexity of $ L^2 $- and $ L^p $-based topologies, we can infer that:
\begin{align}
        &\overline{u}_n \to u \mbox{ and } \underline{u}_n \to u \mbox{ in } L^p(0, T; W_0^{1, p}(\Omega)),\label{conv_20}
        \\[-0.5ex]
        &\mbox{ with }  \bigl| |\nabla \overline{u}_n|_{L^p(0, T; [L^p(\Omega)]^2)}^p -|\nabla \underline{u}_n|_{L^p(0, T; [L^p(\Omega)]^2)}^p \bigr| \leq \frac{2 p E(\alpha^0,u_0)}{\nu} \tau_n \to 0,\nonumber
        \\[1.0ex]
        &\overline{u}_n(t) \to u(t) \mbox{ and } \underline{u}_n(t) \to u(t) \mbox{ in }  W_0^{1, p}(\Omega) , \mbox{ a.e. } t \in (0, T), \mbox{ as $ n \to \infty $.} \label{conv_21}
      \end{align}
      In particular, 
      \begin{align}
        &\overline{u}_n(t) \to u(t) \mbox{ and } \underline{u}_n(t) \to u(t) \mbox{ in }  H_0^{1}(\Omega) , \mbox{ with }\label{conv_22}
        \\
        &\qquad\bigl| \nabla (\overline{u}_n -\underline{u}_n)(t)\bigr|_{[L^2(\Omega)]^2} \leq {\tau_n}^{\frac{1}{2}} |\nabla \partial_t u_n|_{L^2(0, T; [L^2(\Omega)]^2)} \to 0 ,\nonumber
        \\[1.0ex]
        &\bigl| \gamma(R(\overline{\alpha}_n(t)) \nabla \overline{u}_n(t)) -\gamma(R(\alpha(t)) \nabla u(t)) \bigr|_{L^1(\Omega)} \label{conv_23}
        \\
        &\quad\leq 4|\nabla \gamma|_{L^\infty(\R^2; \R^2)}|\nabla\overline{u}_n(t)|_{[L^2(\Omega)]^2}|(\overline{\alpha}_n-\alpha)(t)|_{L^2(\Omega)}\nonumber
        \\
        &\quad+\sqrt{2}|\nabla \gamma|_{L^\infty(\R^2; \R^2)}|\Omega|^{\frac{1}{2}}|\nabla (\overline{u}_n -u)(t)|_{[L^2(\Omega)]^2} \to 0,\nonumber
        \\[1ex]
        &\mbox{ as well as } \bigl| \gamma(R(\underline{\alpha}_n(t)) \nabla \underline{u}_n(t)) -\gamma(R(\alpha(t)) \nabla u(t)) \bigr|_{L^1(\Omega)} \to 0, \nonumber
        \\
        &\hspace{10ex}\mbox{for any $ t \in [0, T] $, as $ n \to \infty $.}\nonumber
\end{align}
On the other hand, if we take $\omega=\alpha_n-\alpha$ in \eqref{conv_10}, then keeping in mind \eqref{conv_1}, \eqref{conv_3} and \eqref{conv_17}, we obtain that
\begin{align}
  &\limsup_{n\rightarrow\infty} 
  \left( 
    \frac{\kappa}{2}\int_0^t|\nabla\overline{\alpha}_n(\sigma)|_{[L^2(\Omega)]^2}^2\,d\sigma
  \right)
  \label{conv_24}
  \\
  &\leq \frac{\kappa}{2}\int_{0}^{t}|\nabla\alpha(\sigma)|^2_{[L^2(\Omega)]^2}\,d\sigma\nonumber
  \\
  &-\lim_{n\rightarrow\infty}\int_0^t\int_\Omega\nabla\gamma(R(\overline{\alpha}_n(\sigma))\nabla\overline{u}_n(\sigma))\cdot R\left(\overline{\alpha}_n(\sigma)+\ts\frac{\pi}{2}\right)\nabla \overline{u}_n(\sigma)(\overline{\alpha}_n-\alpha)(\sigma)\,dxd\sigma
  \nonumber
  \\
  &=\frac{\kappa}{2}\int_{0}^{t}|\nabla\alpha(\sigma)|^2_{[L^2(\Omega)]^2}\,d\sigma.\nonumber
 \end{align}
 Therefore, by (Fact 1) in Section 1, together with \eqref{conv_13} and \eqref{conv_24}, we obtain the following convergences as $n \rightarrow\infty$:
 \begin{align*}
  \int_0^t|\nabla\overline{\alpha}_n(\sigma)|_{[L^2(\Omega)]^2}^2\,d\sigma\rightarrow\int_0^t|\nabla{\alpha}(\sigma)|_{[L^2(\Omega)]^2}^2\,d\sigma, \mbox{ for any $t\in[0,T]$.}
\end{align*}
Moreover, owing to \eqref{f-ene0}, (B-6), and uniform convexity of $L^2$-based topologies, we can infer that:
\begin{align}
  &\overline{\alpha}_n\rightarrow \alpha \mbox{ and } \underline{\alpha}_n\rightarrow\alpha \mbox{ in }L^2(0,T;H^1_0(\Omega)), \label{conv_25}
  \\
  &\quad \mbox{ with } \bigl||\nabla\overline{\alpha}_n|^2_{[L^2(\Omega)]^2}-|\nabla\underline{\alpha}_n|^2_{[L^2(\Omega)]^2}\bigr|\leq \frac{2E(\alpha^0,u_0)}{\kappa}\tau_n\rightarrow0, \mbox{ as }n\rightarrow\infty.\nonumber
\end{align}
\eqref{conv_7}--\eqref{conv_9}, \eqref{conv_20}--\eqref{conv_23}, and \eqref{conv_25} imply 
\begin{equation}\label{conv_26}
  \left\{
    \begin{aligned}
      E_{\varepsilon_n}(\overline{\alpha}_n(t), \overline{u}_n(t)) \to \mathcal{J}_*(t) = E(\alpha(t), u(t)),
      \\
      E_{\varepsilon_n}(\underline{\alpha}_n(t), \underline{u}_n(t)) \to \mathcal{J}_*(t) = E(\alpha(t), u(t)), 
  \end{aligned}
  \mbox{ a.e. $ t \in (0, T) $, as $ n \to \infty $.}\right.
\end{equation}
{Now, setting $\omega=\varphi$ in $H^{1}_0(\Omega)$ in \eqref{conv_10} and $w=\psi$ in $W^{1,p}_0(\Omega)$ in \eqref{conv_11}, and taking into account \eqref{conv_1}, \eqref{conv_3}, \eqref{conv_4}, \eqref{conv_17} and \eqref{conv_18}, we obtain the following by letting $ n \rightarrow \infty $:}
\begin{align}
  &\int_I(\partial_tu(t),u(t)-\psi)_{L^2(\Omega)}\,dt+\lambda\int_I(u(t)-u_\mathrm{org},u(t)-\psi)_{L^2(\Omega)}\,dt
  \nonumber
  \\
  &\quad+\mu\int_I(\nabla\partial_tu(t),\nabla(u(t)-\psi))_{[L^2(\Omega)]^2}\,dt
  \nonumber
  \\
  &\quad+\nu\int_I\int_{\Omega}|\nabla u(t)|^{p-2}\nabla u(t)\cdot\nabla(u(t)-\psi)\,dxdt
  \nonumber
  \\
  &\quad+\int_I\int_{\Omega}\gamma(R(\alpha(t))\nabla u(t))\,dxdt\leq\int_I\int_{\Omega}\gamma (R(\alpha(t))\nabla\psi)\,dxdt,
  \label{conv_27}
\end{align}
for any open interval $I\subset(0,T)$. Also, using \eqref{conv_2}, \eqref{conv_5}, and \eqref{conv_17}, we let $ n \rightarrow \infty $ and obtain that 
\begin{equation}\label{conv_28}
  \kappa\int_I(\nabla{\alpha}(t),\nabla\varphi)_{[L^2(\Omega)]^2} \,dt
  + \int_I\int_\Omega\bm{w}^*(t)\cdot R\left({\alpha}(t)+\ts\frac{\pi}{2}\right)\nabla{u}(t)\varphi\,dxdt= 0,
\end{equation}
for any open interval $I\subset(0,T)$. Furthermore, by using \eqref{subdig1}, \eqref{subdig2} and \eqref{conv_3}, (Fact 3) in Remark \ref{rem2}, one can say that
\begin{align*}
  \bm{w}^*=\partial\widehat{\Phi}^I(R(\alpha)\nabla u)\mbox{ in }L^2(0,T;[L^2(\Omega)]^2),
\end{align*}
and hence,
\begin{equation}\label{subdig3}
  \begin{aligned}
    &\bm{w}^*=\partial\Phi(R(\alpha)\nabla u) \mbox{ in } [L^2(\Omega)]^2, \mbox{ for a.e. }t\in(0,T),\mbox{ and }
    \\
    &\bm{w}^*=\nabla\gamma(R(\alpha)\nabla u) \mbox{ in }\R^2, \mbox{ a.e. in } Q.
  \end{aligned}
 \end{equation}
 From \eqref{conv_27}--\eqref{subdig3}, it follows that the pair $[\alpha,u]$ fulfills (S1) and (S2).

 Next, we consider the energy-inequality. From \eqref{f-ene0}, it is derived that 
 \begin{gather}
  \frac{1}{2}\int_{t_{i-1}}^{t_i}\biggl(|\partial_tu_n(\sigma)|^2_{L^2(\Omega)}+{\mu}|\nabla\partial_tu_n(\sigma)|^2_{[L^2(\Omega)]^2}\biggr)\,d\sigma+E(\overline{\alpha}_n(t),\overline{u}_n(t))\nonumber
  \\
  \leq E(\underline{\alpha}_n(t), \underline{u}_n(t)),~~\mbox{ for }t\in[t_{i-1},t_i),~i=1,2,\dots,{\ts\frac{T}{\tau_n}},\mbox{ and }n\in\N.
  \label{energy1}
\end{gather}
By defining $m^s:=[\frac{s}{\tau}]$ and $m_t:=\bigl([\frac{t}{\tau}]+1\bigr)\land \frac{T}{\tau}$ for $0\leq s< t\leq T$, and summing both sides of \eqref{energy1} for $i=m^s+1, m^s+2,\dots,m_t$, we deduce that
\begin{align}
  &\frac{1}{2}\int_s^t \left(|\partial_tu_n(\sigma)|^2_{L^2(\Omega)}+{\mu}|\nabla\partial_tu_n(\sigma)|^2_{[L^2(\Omega)]^2} \right)\,d\sigma+E_{\varepsilon_n}(\overline{\alpha}_n(t),\overline{u}_n(t))
  \nonumber
  \\
  &\leq \frac{1}{2}\int_{m^s\tau_n}^{m_t\tau_n} \left(|\partial_tu_n(\sigma)|^2_{L^2(\Omega)}+{\mu}|\nabla\partial_tu_n(\sigma)|^2_{[L^2(\Omega)]^2} \right)\,d\sigma+E_{\varepsilon_n}(\overline{\alpha}_n(t),\overline{u}_n(t))
  \nonumber
  \\
  &\leq E_{\varepsilon_n}(\underline{\alpha}_n(s),\underline{u}_n(s)), \mbox{ for }s,t\in[0,T];s\leq t,\mbox{ and }n\in\N.\label{energy2}
\end{align}
Now, thanks to \eqref{conv_1}, \eqref{conv_2}, \eqref{conv_26} and \eqref{energy2}, by letting $n\rightarrow\infty$, we obtain that:
\begin{gather}
  \frac{1}{2}\int_{s}^{t}|\partial_tu(\sigma)|^2_{L^2(\Omega)}\,d\sigma+\frac{\mu}{2}\int_{s}^{t}|\nabla\partial_tu(\sigma)|^2_{[L^2(\Omega)]^2}\,d\sigma
  +E(\alpha(t), u(t)) \leq E(\alpha(s), u(s)),
  \nonumber
  \\
  \mbox{ for a.e. $ s \in [0, T) $ including $s = 0$, and a.e. $ t \in (s, T) $.}\label{energy3}
\end{gather}
Furthermore, we can substitute the phrase ``a.e. $ t \in (s, T) $'' in \eqref{energy3} with ``for any $ t \in [s, T] $''. In fact, let us consider a sequence $\{t_n\}_{n\in\N}\subset (t,T)$ such that $t_n \rightarrow t$,
\begin{gather}
  \frac{1}{2}\int_{s}^{t_n}|\partial_tu(\sigma)|^2_{L^2(\Omega)}\,d\sigma+\frac{\mu}{2}\int_{s}^{t_n}|\nabla\partial_tu(\sigma)|^2_{[L^2(\Omega)]^2}\,d\sigma \nonumber
  \\
  +E(\alpha(t_n), u(t_n)) \leq E(\alpha(s), u(s)),\mbox{ for all $ n\in\N $.}\label{energy4}
\end{gather}
Considering the lower semi-continuity of $E(\alpha,u)$ on $[L^2(\Omega)]^2$ and the convergence $[\alpha(t_n),$ $u(t_n)] \rightarrow [\alpha(t),u(t)]$ in $[L^2(\Omega)]^2$, taking the lower limit of both sides of \eqref{energy4} leads to the following:
\begin{align}
  &\frac{1}{2}\int_{s}^{t}|\partial_tu(\sigma)|^2_{L^2(\Omega)}\,d\sigma+\frac{\mu}{2}\int_{s}^{t}|\nabla\partial_tu(\sigma)|^2_{[L^2(\Omega)]^2}\,d\sigma +E(\alpha(t), u(t))
    \nonumber
    \\
  \leq& \liminf_{n\rightarrow\infty}\Biggl(\frac{1}{2}\int_{s}^{t_n}|\partial_tu(\sigma)|^2_{L^2(\Omega)}\,d\sigma+\frac{\mu}{2}\int_{s}^{t_n}|\nabla\partial_tu(\sigma)|^2_{[L^2(\Omega)]^2}\,d\sigma+E(\alpha(t_n), u(t_n)) \Biggr)
  \nonumber
  \\
  \leq& E(\alpha(s), u(s)). \label{energy5}
\end{align}
In addition, since the uniqueness of solutions to (S) has been established in Subsection \ref{sub4.1}, the energy-inequality \eqref{energy3} holds for all $0 \leq s \leq t \leq T$. More precisely, for the unique solution $[\alpha,u] \in [L^2(0,T;L^2(\Omega))]^2$ to (S) and any $s \in [0,T)$, the uniqueness implies that \eqref{energy3} holds for almost every $\widetilde{s} \in (s,T)$, including $\widetilde{s}=s$, and all $t \in [\widetilde{s},T]$. This follows by applying the time-discretization argument in Section 3 with initial data $[\alpha(s),u(s)] \in H_0^1(\Omega)\times W_0^{1,p}(\Omega)$.

\subsection{Verification of Boundedness: $0\leq u \leq 1$ in $\overline{Q}$}\label{sub4.4}
In this subsection, we establish that {$0 \leq u\leq 1$ in $\overline{Q}$}. Set $\psi= u(t) + [-u(t)]^+$ in (S2). Based on (A1) and (A2), it follows that
\begin{align}
  &\frac{1}{2}\frac{d}{dt}(|[-u(t)]^+|^2_{L^2(\Omega)}+\mu|\nabla [-u(t)]^+|^2_{[L^2(\Omega)]^2})
  +\lambda(u_\mathrm{org},[-u(t)]^+)_{L^2(\Omega)}
  \nonumber
  \\
  &\quad+\int_{\Omega}\gamma(R(\alpha(t))\nabla u(t))\,dx
  \leq\int_{\Omega}\gamma (R(\alpha(t))\nabla[-u(t)]^-)\,dx,\mbox{ a.e. } t \in (0,T).
  \label{ugeq01}
\end{align}
Here, since
\begin{align*}
  \int_{\Omega}\gamma (R(\alpha(t))\nabla[-u(t)]^-)\,dx
  &= \int_{\{u\geq0\}}\gamma (R(\alpha(t))\nabla u(t))\,dx 
  \\
  &\leq\int_{\Omega}\gamma (R(\alpha(t))\nabla u(t))\,dx,
\end{align*}
we derive 
\begin{align}\label{-uve+}
  \frac{d}{dt}\bigl(|[-u(t)]^+|^2_{L^2(\Omega)}+\mu|\nabla [-u(t)]^+|^2_{[L^2(\Omega)]^2}\bigr)\leq0,\mbox{ a.e. } t \in (0,T).
\end{align}
By using Gronwall's lemma to \eqref{-uve+}, it follows that
\begin{align}
  &|[-u(t)]^+|^2_{L^2(\Omega)}+\mu|\nabla [-u(t)]^+|^2_{[L^2(\Omega)]^2}\nonumber
  \\
  &\quad \leq|[-u_0]^+|^2_{L^2(\Omega)}+\mu|\nabla [-u_0]^+|^2_{[L^2(\Omega)]^2}=0,\nonumber
  \\
  &\mbox{ for all }t\in[0,T], \mbox{ which implies $ u \geq 0 $ a.e. in $ Q $.}\label{dai01}
\end{align}

Secondly, by setting $\psi= u\land 1$ in (S2) and observing that $u-(u\land1)=[u-1]^+$, we find that 
\begin{align}
  &(\partial_tu(t),[u-1]^+(t))_{L^2(\Omega)}+\lambda(u(t)-u_\mathrm{org},[u-1]^+(t))_{L^2(\Omega)}
  \nonumber
  \\
  &~~+\mu(\nabla\partial_tu(t),\nabla[u-1]^+(t))_{[L^2(\Omega)]^2}+\nu\int_{\Omega}|\nabla u(t)|^{p-2}\nabla u(t)\cdot\nabla[u-1]^+(t)\,dx
  \nonumber
  \\
  &~~+\int_{\Omega}\gamma(R(\alpha(t))\nabla u(t))\,dx\leq\int_{\Omega}\gamma (R(\alpha(t))\nabla(u\land1)(t))\,dx.
  \label{u-101}
\end{align}
Here, since
\begin{align*}
  &\int_{\Omega}\gamma (R(\alpha(t))\nabla(u\land1)(t))\,dx=\int_{\{u\leq1\}}\gamma (R(\alpha(t))\nabla u(t))\,dx
  \\
  &\hspace*{28ex}\leq\int_{\Omega}\gamma (R(\alpha(t))\nabla u(t))\,dx,
  \\[1ex]
  &\partial_tu=\partial_t(u-1),\mbox{ and }\nabla u=\nabla(u-1),
    \\[1ex]
  &1-u_{\mathrm{org}}\geq0\mbox{ a.e. in }\Omega, \mbox{ i.e., }(1-u_{\mathrm{org}},[u-1]^+(t))_{L^2(\Omega)}\geq0,
    \\[1ex]
  &\nu\int_{\Omega}|\nabla u(t)|^{p-2}\nabla u(t)\cdot\nabla[u-1]^+(t)\,dx\geq \nu\int_{\{u\geq1\}}|\nabla(u-1)|^p\,dx \geq0,
\end{align*}
we derive 
\begin{equation}\label{1-uve+}
  \frac{d}{dt}\bigl(|[u(t)-1]^+|_{L^2(\Omega)}^2+\mu|\nabla[u(t)-1]^+|_{[L^2(\Omega)]^2}^2\bigr)\leq0,\mbox{ a.e. } t \in (0,T).
\end{equation}
By using Gronwall's lemma to \eqref{1-uve+}, it follows that
\begin{align}
  &|[u(t)-1]^+|_{L^2(\Omega)}^2+\mu|\nabla[u(t)-1]^+|_{[L^2(\Omega)]^2}^2\nonumber
  \\
  &\qquad\leq|[u_0-1]^+|_{L^2(\Omega)}^2+\mu|\nabla[u_0-1]^+|_{[L^2(\Omega)]^2}^2=0,\nonumber
  \\
  &\mbox{ for all }t\in[0,T], \mbox{ which implies $ u \leq 1 $ a.e. in $ Q $.}\label{dai02}
\end{align}
By integrating \eqref{dai01}, \eqref{dai02}, and utilizing the continuity property $u\in C(\overline{Q})$, we deduce that $0\leq u\leq1 $ in $\overline{Q}$. 
\medskip

Thus, by assuming $\kappa_*$ to be $\widehat{\kappa}$, all assertions of Main Theorem follow.
\qed









\end{document}